\documentclass[a4paper,11pt,DIV=11,oneside,headings=small,bibliography=totoc,abstract=true]{scrartcl}
\usepackage[linktocpage, colorlinks]{hyperref}
	\usepackage{color}
 	\definecolor{darkred}{rgb}{0.5,0,0}
 	\definecolor{darkgreen}{rgb}{0,0.5,0}
 	\definecolor{darkblue}{rgb}{0,0,0.5}
 	\hypersetup{colorlinks,linkcolor=darkblue,filecolor=darkgreen,urlcolor=darkred,citecolor=darkblue}
\hypersetup{
pdftitle={Sufficient criteria and sharp geometric conditions for observability in Banach spaces},
pdfauthor={Dennis Gallaun, Christian Seifert, Martin Tautenhahn},
}
\usepackage{amssymb,amsmath,amsfonts,amsthm}
\allowdisplaybreaks
\usepackage{mathtools}
\usepackage[english]{babel}
\usepackage[T1]{fontenc}
\usepackage[shortlabels]{enumitem}
\usepackage[noblocks]{authblk}

\newcommand{\ball}[2]{B (#1 , #2)}

\newcommand{\thickset}{E}
\newcommand{\ii}{\mathrm{i}}

\newcommand{\drm}{\mathrm{d}}
\newcommand{\euler}{\mathrm{e}}

\newcommand{\N}{\mathbb{N}}
\newcommand{\R}{\mathbb{R}}
\newcommand{\C}{\mathbb{C}}

\newcommand{\cL}{\mathcal{L}}
\newcommand{\F}{\mathcal{F}}

\newcommand{\dom}{\mathcal{D}}
\newcommand{\1}{\mathbf{1}}
\newcommand{\from}{\colon}
\newcommand{\re}{\operatorname{Re}}
\newcommand{\diverg}{\operatorname{div}}
\newcommand{\grad}{\operatorname{grad}}
\newcommand{\supp}{\operatorname{supp}}
\newcommand{\id}{\operatorname{Id}}
\newcommand{\ran}{\operatorname{Ran}}
\renewcommand{\epsilon}{\varepsilon}
\newtheorem{theorem}{Theorem}[section]
\newtheorem{lemma}[theorem]{Lemma}
\newtheorem{corollary}[theorem]{Corollary}

\theoremstyle{definition}
\newtheorem{example}[theorem]{Example}
\newtheorem{definition}[theorem]{Definition}

\theoremstyle{remark}
\newtheorem{remark}[theorem]{Remark}

\title{Sufficient criteria and sharp geometric conditions for observability in Banach spaces}

\author{Dennis Gallaun}
\affil{Technische Universit\"at Hamburg, Institut f\"ur Mathematik, Am Schwarzenberg-Campus 3, 21073 Hamburg, Germany, \{dennis.gallaun, christian.seifert\}@tuhh.de}

\author[1]{Christian Seifert}

\author{Martin Tautenhahn}
\affil{Technische Universit\"at Chemnitz, Fakult\"at f\"ur Mathematik, 09107 Chemnitz, Germany, martin.tautenhahn@mathematik.tu-chemnitz.de}

\date{\vspace{-7ex}}

\begin{document}

\maketitle

\begin{abstract}
Let $X,Y$ be Banach spaces, $(S_t)_{t \geq 0}$ a $C_0$-semigroup on $X$, $-A$ the corresponding infinitesimal generator on $X$, $C$ a bounded linear operator from $X$ to $Y$, and $T > 0$. We consider the system
\[
  \dot{x}(t)  = -Ax(t), \quad y(t) = Cx(t), \quad t\in (0,T], \quad  x(0) = x_0 \in X.
\]
We provide sufficient conditions such that this system satisfies a final state observability estimate in $L_r ((0,T) ; Y)$, $r \in [1,\infty]$. These sufficient conditions are given by an uncertainty relation and a dissipation estimate. Our approach unifies and generalizes the respective advantages from earlier results obtained in the context of Hilbert spaces.
As an application we consider the example where $A$ is an elliptic operator in $L_p(\mathbb{R}^d)$ for $1<p<\infty$ and where $C = \mathbf{1}_\thickset$ is the restriction onto a thick set $\thickset \subset \mathbb{R}^d$. In this case, we show that the above system satisfies a final state observability estimate if and only if $\thickset \subset \mathbb{R}^d$ is a thick set. Finally, we make use of the well-known relation between observability and null-controllability of the predual system and investigate bounds on the corresponding control costs.
\\[1ex]
\\[1ex]
\end{abstract}
\section{Introduction}
Let $X,Y$ be Banach spaces, $(S_t)_{t \geq 0}$ a $C_0$-semigroup on $X$, $-A$ the corresponding infinitesimal generator on $X$, and $C$ a bounded operator from $X$ to $Y$. We consider systems of the form
\begin{equation} \label{eq:system:obs:intro}
\begin{aligned}
  \dot{x}(t) & = -Ax(t), \quad & &t\in (0,T] ,\quad x(0)  = x_0 \in X, \\
  y(t)  &= Cx(t), \quad   & & t\in [0,T] ,
  \end{aligned}
\end{equation}
where $T>0$ can be thought of as a final time for the system. One interpretation of the second equation in \eqref{eq:system:obs:intro} is that we cannot measure the state $x(t)$ at time $t$ directly, but just some $y(t) = C x(t)$ from the range of $C$. 
The focus of this paper relates to the question whether the system~\eqref{eq:system:obs:intro} satisfies a final state observability estimate in $L_r ((0,T);Y)$ with $r \in [1,\infty]$, that is, there exists $C_{\mathrm{obs}} > 0$ such that for all $x_0 \in X$ we have 
$
\lVert x (T) \rVert_X \leq C_{\mathrm{obs}} \lVert y \rVert_{L_r ((0,T) ; Y)} 
$.
A final state observability estimate thus allows one to recover information on the final state $x(T)$ from suitable measurements $y(t)$ for $t \in (0,T)$. 
\par
The most studied example of the system \eqref{eq:system:obs:intro} is the heat equation with heat generation term in $L_2 (\Omega)$ with $\Omega \subset \R^d$ open and some nonempty observability set, i.e.\ $A = \Delta - V$ is a self-adjoint Schr\"odinger operator in $L_2 (\Omega)$ with bounded potential $V$, and $C = \1_\thickset$ is the projection onto some non-empty measurable set $\thickset \subset \Omega$. For bounded domains $\Omega \subset \R^d$ the observability problem for the heat equation is well understood since the seminal works by Lebeau and Robbiano \cite{LebeauR-95} and Fursikov and Imanuvilov \cite{FursikovI-96}.
For unbounded domains this problem has been studied, e.g., in \cite{Miller-05,Miller-05b,GonzalezT-07,Barbu-14}. While for bounded domains it is sufficient that $\thickset$ is open and nonempty, or even measurable with positive Lebesgue measure \cite{ApraizEWZ-14,EscauriazaMZ-15}, this is of course not true for unbounded domains. On unbounded domains a sufficient geometric condition for observability is given in \cite{LeRousseauM-16}. In addition to that, the papers \cite{EgidiV-18,WangWZZ-19} show that the free heat equation in $L_2 (\R^d)$ satisfies a final state observability estimate if and only if $\thickset$ is a thick set. 
\par
Since the observability constant $C_{\mathrm{obs}}$ can be interpreted (by duality) as the cost for the corresponding null-controllability problem, the problem of obtaining explicit bounds on $C_{\mathrm{obs}}$ attracted particular attention in the literature. The (optimal) dependence of $C_{\mathrm{obs}}$ on the model parameter $T$ is investigated in \cite{Guichal-85,FernandezZ-00,Miller-04b,Phung-04,Miller-06,Miller-06b,TenenbaumT-07,Miller-10,LeRousseauL-12,BeauchardP-18}, while \cite{Miller-04b, TenenbaumT-11, ErvedozaZ-11, NakicTTV-18, EgidiV-18, Phung-18,Egidi-18-arxiv,LaurentL-18-arxiv,NakicTTV-18b-arxiv} also study the dependence on the geometry of the control set $\thickset$.
Moreover, \cite{Guichal-85,Miller-06, TenenbaumT-07,Lissy-12,Lissy-15,DardeE-19} concern one-dimensional problems and boundary control.
\par
One possible approach to show an observability estimate has been described in the papers \cite{LebeauR-95,LebeauZ-98,JerisonL-99}, that is, to prove a quantitative uncertainty relation for spectral projectors. This is an inequality of the type
\begin{equation*}
 \forall \lambda > 0 \ \forall \psi \in L_2 (\Omega) \colon \quad 
 \lVert P (\lambda) \psi \rVert_{L_2 (\Omega)}
 \leq
 d_0 \euler^{d_1 \lambda^{\gamma}}
 \lVert \1_\thickset P (\lambda) \psi \rVert_{L_2 (\thickset)} ,
\end{equation*}
where $\gamma \in (0,1)$, $d_0,d_1 > 0$, and where $P (\lambda)$ denotes the projector to the spectral subspace of $-\Delta + V$ below $\lambda$. Subsequently, this strategy is generalized to (contraction) semigroups in abstract Hilbert spaces with (possibly self-adjoint) generators $-A$, to name those which are closest related to our result; see \cite{Miller-10,TenenbaumT-11,WangZ-17,BeauchardP-18,NakicTTV-18b-arxiv}. In particular, the papers \cite{Miller-10,WangZ-17,BeauchardP-18} allow for the $P (\lambda)$ to be arbitrary projectors (onto semigroup invariant subspaces) by assuming additionally a so-called dissipation estimate, that is, a decay estimate of the semigroup on the orthogonal complement of the range of $P(\lambda)$. This can be rephrased in a scheme in which an uncertainty relation together with a dissipation estimate implies an observability estimate.
Since the constants appearing in the uncertainty relation (and the dissipation estimate) transfer into the observability constant $C_{\mathrm{obs}}$, it is important to achieve its dependence on $d_0$, $d_1$, and $\gamma$, on the set $\thickset$, and on the coefficients of the operator $A$ as explicitly as possible. 
Uncertainty relations with an explicit dependence on the geometry of $\thickset$ 
are provided by the Logvinenko--Sereda theorem for the free heat equation observed on thick sets \cite{LogvinenkoS-74,Kovrijkine-00,Kovrijkine-01,EgidiV-16-arxiv}. For Schr\"odinger operators such uncertainty relations have, for instance, been proven in \cite{NakicTTV-18,NakicTTV-18-arxiv} for a certain class of equidistributed observation sets and bounded potentials and in \cite{LebeauM-19-arxiv} for thick observation sets and analytic potentials.
\par
So far, the discussion has been restricted to Hilbert spaces only. However, a natural setup to ask for observability estimates is the context of Banach spaces and $C_0$-semigroups, since there are various applications of the above concepts in this situation.
In this paper, we extend (some of) the above-mentioned results to the Banach space setting. In particular, in Section~\ref{sec:abstract_obs_estimate} we show in the general framework of Banach spaces that an uncertainty relation together with a dissipation estimate implies that the system~\eqref{eq:system:obs:intro} satisfies a final state observability estimate. Our observability constant $C_{\mathrm{obs}}$ is given explicitly with respect to the parameters coming from the uncertainty relation and the dissipation estimate and, in addition, is sharp in the dependence on $T$.
Let us stress that, besides the fact that this result holds in its natural Banach space setting, our
approach unifies and generalizes the respective advantages from earlier results even in the context of Hilbert spaces; cf.\ Remark~\ref{rem:UR-OBS} for more details. In Section~\ref{sec:L_p(Rd)} we verify these sufficient conditions in $L_p$-spaces for a class of elliptic operators $A$ and observation operators $C = \1_\thickset$. This way we obtain an observability estimate with an explicit dependence on the coefficients of the elliptic operator $A$, the final time $T$, and the geometry of the thick set $\thickset$. Furthermore, we show that this result is sharp in the sense that the system~\eqref{eq:system:obs:intro} satisfies a final state observability estimate if and only if $\thickset$ is a thick set. 
Finally, in Section~\ref{sec:Control_Costs} we make use of the well-known relation between observability and null-controllability of the predual system to \eqref{eq:system:obs:intro} and investigate bounds on the corresponding control costs.
\section{Sufficient criteria for observability in Banach spaces}
\label{sec:abstract_obs_estimate}
For normed spaces $V$ and $W$ we denote by $\cL (V , W)$
the space of bounded linear operators from $V$ to $W$.
Let $X,Y$ be Banach spaces, $(S_t)_{t \geq 0}$ a $C_0$-semigroup on $X$, $-A$ the corresponding infinitesimal generator on $X$ with domain $\dom (-A)$, and $C \in \cL (X,Y)$. For $T>0$ we consider the system
\begin{equation} \label{eq:system:obs}
\begin{aligned}
  \dot{x}(t) & = -Ax(t), \quad & &t\in (0,T] ,\quad x(0)  = x_0 \in X, \\
  y(t)  &= Cx(t), \quad   & & t\in [0,T] .
  \end{aligned}
\end{equation}
The mild solution of \eqref{eq:system:obs} is given by 
\[
 x (t) = S_t x_0, \quad y(t) = C S_t x_0, \quad t\in [0,T] .
\]
In particular, if $x_0 \in \dom (A)$ we may differentiate $x (\cdot) = S_{(\cdot)} x_0$ to obtain \eqref{eq:system:obs}. Let $r \in [1,\infty]$. We say that the system \eqref{eq:system:obs} satisfies a \emph{final state observability estimate in $L_r ((0,T) ; Y)$} if there exists $C_{\mathrm{obs}} > 0$ such that for all $x_0 \in X$ we have $\lVert x (T) \rVert_X \leq C_{\mathrm{obs}} \lVert y \rVert_{L_r ((0,T) ; Y)}$ or, equivalently, if for all $x_0\in X$ we have
\begin{align*}
  \lVert S_T x_0 \rVert_X 
  & \leq C_{\mathrm{obs}} \left( \int_0^T \lVert C S_\tau x_0 \rVert_Y^r \drm \tau \right)^{1/r} &&\quad\text{if $1\leq r<\infty$ or} \\
  \lVert S_T x_0 \rVert_X & \leq C_{\mathrm{obs}} \operatorname*{ess\,sup}_{\tau \in[0, T]}\lVert C S_\tau x_0 \rVert_Y &&\quad\text{if $r=\infty$}.
\end{align*}
One motivation to study final state observability estimates is their relation to null-\-con\-trol\-lability of the predual system to \eqref{eq:system:obs} and its control cost. This is discussed in more detail in Section~\ref{sec:Control_Costs}.
\par
The following theorem provides sufficient conditions such that the system \eqref{eq:system:obs} satisfies a final state observability estimate.
\begin{theorem} \label{thm:spectral+diss-obs} Let $X$ and $Y$ be Banach spaces, $C\in \cL(X,Y)$, $(S_t)_{t\geq 0}$ be a $C_0$-semigroup on $X$, $M \geq 1$ and $\omega \in \R$ such that $\lVert S_t \rVert \leq M \euler^{\omega t}$ for all $t \geq 0$, $\lambda^* \geq 0$ and $(P_\lambda)_{\lambda>\lambda^*}$ be a family of bounded linear operators in $X$.
Assume further that there exist $d_0,d_1,\gamma_1 > 0$ such that
\begin{align} 
\forall x\in X \ \forall \lambda > \lambda^* &\colon \quad \lVert P_\lambda x \rVert_{ X } \le d_0 \euler^{d_1 \lambda^{\gamma_1}} \lVert C  P_\lambda x \rVert_{Y }
\label{eq:ass:uncertainty}
\end{align}
and that there exist $d_2\geq 1$, $d_3,\gamma_2,\gamma_3, T>0$ with $\gamma_1<\gamma_2$ such that
\begin{align}
\forall x\in X \ \forall \lambda > \lambda^* \ \forall t\in (0,T/2] &\colon \quad \lVert (\id-P_\lambda) S_t x \rVert_{X} \le d_2 \euler^{-d_3 \lambda^{\gamma_2} t^{\gamma_3}} \lVert x \rVert_{X}. \label{eq:ass:dissipation}
\end{align}
Then we have for all $r\in[1,\infty]$ and $x \in X$
\begin{align*}
 \lVert S_T x \rVert_{X} \leq C_{\mathrm{obs}}   \lVert CS_{(\cdot)}x \rVert_{L_r ((0,T);Y)}
 \quad\text{with}\quad
 C_{\mathrm{obs}} = \frac{C_1}{T^{1/r}} \exp \left(\frac{C_2}{T^{\frac{\gamma_1 \gamma_3}{\gamma_2 - \gamma_1}}} +   C_3 T\right), 
\end{align*}
where $T^{1/r} = 1$ if $r=\infty$, and
\begin{align*}
  C_1 
  &= (4 M d_0) \max \Bigl\{\left( (4d_2 M^2)   (d_0 \lVert C \rVert_{\cL (X,Y)}+1) \right)^{8/(\euler \ln 2)}, \euler^{4d_1\left(2\lambda^*\right)^{\gamma_1}}\Bigr\}, \\
 C_2 &= 4 \bigl(2^{\gamma_1} (2\cdot 4^{\gamma_3})^\frac{\gamma_1 \gamma_2}{\gamma_2-\gamma_1} d_1^{\gamma_2} / d_3^{\gamma_1} \bigr)^{\frac{1}{\gamma_2-\gamma_1}} , \\[1ex]
 C_3 & = \max\{\omega , 0\} \bigl(1 + 10 / (\euler \ln 2) \bigr).
\end{align*}
\end{theorem}
\begin{remark} \label{rem:UR-OBS} 
\begin{enumerate}[label=(\alph*),wide,labelindent=\parindent,parsep=0cm]
 \item Assumption~\eqref{eq:ass:uncertainty} is called \emph{uncertainty relation}, since a state $P_\lambda x \not = 0$ in the range of $P_\lambda$ cannot be in the kernel of $C$. In particular, if $X = Y$ is a Hilbert space and $P_\lambda$ and $C$ are orthogonal projections, assumption~\eqref{eq:ass:uncertainty} can be rewritten as 
\begin{equation} \label{eq:uncertainty_hilbertspace}
 P_\lambda \leq d_0 \euler^{d_1 \lambda^{\gamma_1}} P_\lambda C P_\lambda ,
\end{equation}
where the inequality is understood in the quadratic form sense. If $X=Y = L_2 (\Omega)$ with $\Omega\subset \R^d$ open, $A$ is a Schr\"odinger operator, $C = \1_\thickset\from X\to Y$ is the restriction operator (i.e.\ the multiplication operator with $\1_\thickset$) on some measurable set $\thickset \subset \Omega$, and if $P_\lambda$ is the spectral projector of a self-adjoint operator onto the interval $(-\infty , \lambda]$, then the spectral projector corresponds to a restriction in momentum-space and enforces delocalization in direct space, i.e., an uncertainty relation. Inequality \eqref{eq:uncertainty_hilbertspace} is sometimes also called \emph{gain of positive definiteness}, since the restriction $P_\lambda C P_\lambda$ of $C$ is strictly
positive on the subspace $\ran P_\lambda$. In control theory inequalities of the type \eqref{eq:ass:uncertainty} are often called \emph{spectral inequality}. We omit this notation, since the operators $P_\lambda$ are in our setting not necessarily spectral projectors of some self-adjoint operator in a Hilbert space.
\par
Assumption~\eqref{eq:ass:dissipation} is called \emph{dissipation estimate}, as it assumes an exponential decay of $(\id - P_\lambda) S_t$ with respect to $\lambda$ and $t$. In particular, it implies that $P_\lambda\to \id$ strongly as $\lambda\to \infty$.
\item\label{rem:UR-OBS:2} The dependence of $C_{\mathrm{obs}}$ on $T$ is optimal for large and small $T$. In \cite{Seidman-84} Seidman showed for one-dimensional controlled heat systems that $C_{\mathrm{obs}}$ blows up at most exponentially for small $T$. This result was extended to arbitrary dimension by Fursikov and Imanuvilov in \cite{FursikovI-96}. That the exponential blow-up has to occur for small $T$ was first shown by G\"uichal \cite{Guichal-85} for one-dimensional systems and by Miller \cite{Miller-04} in arbitrary dimension. It is folklore that in the large time regime, the decay rate $T^{-1/r}$ is optimal; for a proof see, e.g., \cite[Theorem~2.13]{NakicTTV-18b-arxiv}.
\item Let us discuss the novel aspects of Theorem~\ref{thm:spectral+diss-obs} compared to earlier results in the literature. 
 We restrict our discussion to the case where $X$ and $Y$ are Hilbert spaces and $r=2$, since to the best of our knowledge, sufficient conditions for observability in Banach spaces as in Theorem~\ref{thm:spectral+diss-obs} have not been obtained before. 
\par
 That uncertainty relations imply observability estimates was first shown in the seminal papers \cite{LebeauR-95,LebeauZ-98,JerisonL-99}. Subsequently, there is a huge amount of literature concerning abstract theorems which turn uncertainty relations into observability estimates in Hilbert spaces; to name a few, see \cite{Miller-10,TenenbaumT-11,WangZ-17,BeauchardP-18,NakicTTV-18b-arxiv}. The paper \cite{Miller-10} considered general $C_0$-semigroups and the operators $(P_\lambda)_{\lambda > 0}$ as projections onto a nondecreasing family of semigroup invariant subspaces. The obtained observability constant $C_{\mathrm{obs}}$ is of the form $C \exp (C / T^{\gamma_1 \gamma_3 / (\gamma_2 - \gamma_1)})$ and hence misses the factor $T^{-1/2}$. A similar result has been obtained in \cite{BeauchardP-18} for contraction semigroups and orthogonal projections $(P_\lambda)_{\lambda > 0}$ onto semigroup invariant subspaces. The papers \cite{TenenbaumT-11,NakicTTV-18b-arxiv} considered nonnegative and self-adjoint operators $A$, and the operators $(P_\lambda)_{\lambda > 0}$ are assumed to be spectral projections of $A$ onto the interval $[0,\lambda)$. In this setting, the dissipation estimate is automatically satisfied with $\gamma_2 = \gamma_3 = 1$. While both papers obtain the ``optimal'' bound $C T^{-1/2} \exp (C / T^{\gamma_1/ (1 - \gamma_1)})$ (including the factor $T^{-1/2}$), the paper \cite{TenenbaumT-11} assumed additionally that $A$ has purely discrete spectrum with an orthogonal basis of eigenvectors. Moreover, \cite{NakicTTV-18b-arxiv} slightly improved the dependence of $C_{\mathrm{obs}}$ on the parameters $d_0$ and $d_1$ which was essential for their application to certain homogenization regimes. Let us emphasize that our result recovers this dependence on $d_0$ and $d_1$ as well and hence allows also for homogenization.
\par
 To conclude, our result extends the earlier mentioned results into three directions.
 \begin{enumerate}
  \item[(1)] We allow for an arbitrary family $(P_\lambda)_{\lambda > \lambda^*}$ of bounded linear operators, and obtain at the same time the factor $T^{-1/2}$ in $C_{\mathrm{obs}}$. In particular, we do not require that $P_\lambda$ is an orthogonal projection. 
  \item[(2)] We allow for general $C_0$-semigroups, possibly with exponential growth. We do not require contraction (or quasi-contraction, or bounded) semigroups.
 \end{enumerate}
As in \cite{WangZ-17}, our result combines the respective advantages from earlier results in Hilbert space setting, e.g., the factor $T^{-1/2}$ in $C_{\mathrm{obs}}$, general $C_0$-semigroups, and arbitrary family $(P_\lambda)_{\lambda > \lambda^*}$ of bounded linear operators at the same time.
 \begin{enumerate}
  \item[(3)] We consider Banach spaces $X$ and $Y$ instead of Hilbert spaces and $r \in [1,\infty]$ instead of $r = 2$.
 \end{enumerate}
 Indeed, since the theory of strongly continuous semigroups essentially is a Banach space theory, our Theorem~\ref{thm:spectral+diss-obs} now formulates the link from uncertainty relations and dissipation estimates to observability estimates in its natural setup. Let us stress that, in contrast to the above-mentioned references, we have no spectral calculus in the general framework of Banach spaces.
\item Suppose we have a discrete sequence $(P_k)_{k\in \N}$ of bounded linear operators in $X$ which satisfies the following discrete version of conditions \eqref{eq:ass:uncertainty} and \eqref{eq:ass:dissipation}:
	\begin{equation*} 
\forall x\in X \ \forall k \in \N \colon \quad \lVert P_k x \rVert_{ X } \le \tilde{d}_0 {\euler}^{\tilde{d}_1 k^{\gamma_1}} \lVert C  P_k x \rVert_{Y }
\end{equation*}
and
\begin{equation*}
\forall x\in X \ \forall k \in \N \ \forall t\in (0,T/2] \colon \quad \lVert (1-P_k) S_t x \rVert_{X} \le \tilde{d}_2 {\euler}^{-\tilde{d}_3 k^{\gamma_2} t^{\gamma_3}} \lVert x \rVert_{X}
\end{equation*}
for constants $\tilde{d}_0, \tilde{d}_1, \tilde{d}_3 >0$ and $\tilde{d}_2\geq 1$, as in \cite{BeauchardP-18}. Then we can apply Theorem~\ref{thm:spectral+diss-obs} in the following way. Let $(P_\lambda)_{\lambda>0}$ be defined by $P_\lambda = P_k$ for $\lambda \in (k-1,k]$, $k\in \N$. Then $(P_\lambda)_{\lambda>0}$ fulfils the assumptions \eqref{eq:ass:uncertainty} and \eqref{eq:ass:dissipation} of Theorem~\ref{thm:spectral+diss-obs} with
\begin{equation*}
d_0 := \tilde{d}_0\euler^{\tilde{d}_1}, \quad d_1 := 2^{\gamma_1} \tilde{d}_1, \quad d_2 := \tilde{d}_2 \quad  d_3 := \tilde{d}_3, \quad \text{and} \quad \lambda^* = 0.
\end{equation*}
\item It is possible to extend the statement of Theorem~\ref{thm:spectral+diss-obs} to the case of time-dependent observation operators $C\from [0,T]\to \cL(X,Y)$, as long as $C$ is measurable (in a suitable sense) and essentially bounded. We refer to \cite[Theorem~2.11]{NakicTTV-18b-arxiv} for a similar extension in Hilbert spaces.
\item
It is also possible to prove an interpolation inequality as in \cite[Theorem 6]{ApraizEWZ-14} for our abstract context. This can then be used to obtain a final state observability estimate by only taking into account the observation function on a measurable subset of the time interval $[0,T]$; cf.\ \cite[Theorem 1.1]{PhungW-13}.
\end{enumerate}
\end{remark}
\begin{proof}[Proof of Theorem~\ref{thm:spectral+diss-obs}]
 Assume we have shown the statement of the theorem in the case $r=1$, i.e., for all $x \in X$ we have
 \[
 \lVert S_T x \rVert_{X} \leq C_{\mathrm{obs}}   \lVert CS_{(\cdot)}x \rVert_{L_1 ((0,T);Y)}, \quad\text{where}\quad
 C_{\mathrm{obs}}  = \frac{C_1}{T} \exp \left(\frac{C_2}{T^{\frac{\gamma_1 \gamma_3}{\gamma_2 - \gamma_1}}} +   C_3 T\right) .
 \]
 Then, by H\"older's inequality we obtain for all $r \in [1,\infty]$ and all $x \in X$
 \[
  \lVert S_T x \rVert_{X} \leq C_{\mathrm{obs}} T^{1/r'} \lVert CS_{(\cdot)}x\rVert_{L_{r}((0,T);Y)} ,
 \]
 where $r' \in [1,\infty]$ is such  that $1/r + 1/r' = 1$. Since $T^{-1} T^{1/r'} = T^{-1/r}$, the statement of the theorem follows. Thus, it is sufficient to prove the theorem in the case $r = 1$.
 \par
 For the first part of the proof, we adapt the strategy in \cite{TenenbaumT-11, NakicTTV-18b-arxiv} with a slight modification in order to deal with general $P_\lambda$'s instead of spectral projectors.
Fix $x \in X$ arbitrary, and introduce for $t > 0$ and $\lambda>\lambda^*$ the notation
 \begin{align*}
   F (t) &= \bigl\lVert S_t x \bigr\rVert_X ,  
 & F_\lambda (t) &= \bigl\lVert P_\lambda S_t x \bigr\rVert_X ,
 & F_\lambda^\perp (t) &= \bigl\lVert (\id - P_\lambda) S_t x \bigr\rVert_X , \\
   G (t) &= \bigl\lVert C S_t x \bigr\rVert_Y ,
 & G_\lambda (t) &= \bigl\lVert C P_\lambda S_t x \bigr\rVert_Y ,
 & G_\lambda^\perp (t) &=\bigl\lVert C (\id - P_\lambda) S_t  x \bigr\rVert_Y .
 \end{align*}
 Then for $0\leq \tau\leq t$ we obtain
 \[
 F(t) 
 = 
 \lVert S_t x \rVert_X 
 = 
 \lVert S_{t-\tau}S_\tau x \rVert_X 
 \leq 
 M\euler^{\omega_+ t} \lVert S_\tau x \rVert_X
 = 
 M \euler^{\omega_+ t} F(\tau) ,
 \]
 where $\omega_+ = \max\{\omega , 0\}$. Integrating this inequality, we obtain
 \[
 F(t) \leq M\euler^{\omega_+ t} \frac{2}{t} \int_{t/2}^t F(\tau)\drm \tau.
 \] 
 We now use the uncertainty relation \eqref{eq:ass:uncertainty} to obtain for all $t > 0$ and $\lambda>\lambda^*$
 \begin{align*}
  F (t) 
  \leq 
  M\euler^{\omega_+ t}\frac{2}{t} \int_{t/2}^t F (\tau) \drm \tau 
  &\leq 
  M\euler^{\omega_+ t}\frac{2}{t} \int_{t/2}^t \left( F_\lambda (\tau) + F_\lambda^\perp (\tau) \right) \drm \tau \\
  &\leq
  M\euler^{\omega_+ t}\frac{2}{t} \int_{t/2}^t \left( d_0 \euler^{d_1 \lambda^{\gamma_1}} G_\lambda (\tau) + F_\lambda^\perp (\tau) \right) \drm \tau .
 \end{align*}
 By the semigroup property and the dissipation estimate \eqref{eq:ass:dissipation} we have for all $\tau \in (0,T]$ the estimate
 \begin{equation} \label{eq:dissi-trick}
 F_\lambda^\perp (\tau) = \lVert (\id - P_\lambda) S_{\tau / 2} S_{\tau / 2} x \rVert_X \leq d_2 \euler^{-d_3 \lambda^{\gamma_2} (\tau/2)^{\gamma_3}} F (\tau / 2). 
 \end{equation}
Since $F(\tau/2) \leq M\euler^{\omega_+ t / 4} F(t/4)$ for $t > 0$ and $\tau \in [t/2 , t]$, we obtain for all $t \in (0,T]$ and $\lambda>\lambda^*$
 \begin{equation} \label{eq:1}
  F (t) 
\le \frac{2 M\euler^{\omega_+ t} d_0 \euler^{d_1\lambda^{\gamma_1}}}{t} \int_{t/2}^t G_\lambda(\tau) \drm \tau + d_2 M^{2}\euler^{5 \omega_+ t / 4}  \euler^{-  d_3 \lambda^{\gamma_2} (t/4)^{\gamma_3}} F(t / 4) .
 \end{equation}
 Using $G_\lambda (\tau) \leq G (\tau) + G_\lambda^\perp (\tau) \leq  G (\tau) + \lVert C \rVert_{\cL (X,Y)} F_\lambda^\perp (\tau)$ and \eqref{eq:dissi-trick} again, we obtain for all $t \in (0,T]$ and $\lambda>\lambda^*$
 \begin{equation} \label{eq:2}
  \int_{t/2}^t G_\lambda(\tau) \drm \tau
  \leq 
  \int_{t/2}^t G(\tau) \drm \tau + \lVert C \rVert_{\cL (X,Y)} d_2 \euler^{-d_3 \lambda^{\gamma_2} (t / 4)^{\gamma_3} } \int_{t/2}^t  F (\tau / 2) \drm \tau .
 \end{equation}
 Since $F (\tau / 2) \leq M\euler^{\omega_+ t/4} F (t / 4)$ for $t > 0$ and $\tau \in [t/2 , t]$ and $1 \leq \euler^{d_1\lambda^{\gamma_1}}$, we conclude from \eqref{eq:1} and \eqref{eq:2} for all $t \in (0,T]$ and $\lambda>\lambda^*$
 \begin{align*} 
F (t) 
  & \leq \frac{2M d_0 \euler^{d_1\lambda^{\gamma_1}}}{t \euler^{-\omega_+ T}} \int_{t/2}^t G(\tau) \drm \tau +  \frac{d_2 M^{2} \euler^{5\omega_+ T / 4} \euler^{d_1\lambda^{\gamma_1}}}{\euler^{d_3 \lambda^{\gamma_2} (t/4)^{\gamma_3}}}\left(  d_0 \lVert C \rVert_{\cL (X,Y)} +1 \right) F(t / 4) .
 \end{align*}
With the short hand notation
\[
 D_1 (t,\lambda) = \frac{2M\euler^{\omega_+ T} d_0 \euler^{d_1\lambda^{\gamma_1}}}{t} \int_{t/2}^t G(\tau) \drm \tau,
 \quad
 D_2 (t,\lambda) = K_1 \euler^{d_1\lambda^{\gamma_1}- d_3 \lambda^{\gamma_2} (t/4)^{\gamma_3}} ,
\]
where $K_1 = (d_0 \lVert C \rVert_{\cL (X,Y)}+1) d_2 M^{2}\euler^{5\omega_+ T / 4}$, this can be rewritten as
\begin{equation} \label{eq:before_iterating}
 F (t) \le D_1 (t,\lambda) + D_2 (t,\lambda) F (t/4) .
\end{equation}
This inequality can be iterated. Let $(\lambda_k)_{k \in \N_0}$ be a sequence with $\lambda_k > \lambda^*$ for $k \in \N_0$.
First we apply inequality~\eqref{eq:before_iterating} with $t = T$ and $\lambda = \lambda_0$.
The term $F (4^{-1}T)$ on the right-hand side is then estimated by inequality~\eqref{eq:before_iterating} with $t = 4^{-1} T$ and $\lambda = \lambda_1$. This way, we obtain after two steps
\begin{align*}
 F (T) &\leq  D_1 (T , \lambda_0) + D_2 (T , \lambda_0) \left(  D_1 (4^{-1} T , \lambda_1) + D_2 (4^{-1}T , \lambda_1) F (4^{-2}T) \right) \\
 & =  D_1 (T , \lambda_0) + D_1 (4^{-1}T , \lambda_1) D_2 (T , \lambda_0) + D_2 (T , \lambda_0) D_2 (4^{-1}T , \lambda_1) F (4^{-2}T ) .
\end{align*}
After $N + 1$ steps of this type we obtain
\begin{multline} \label{eq:after_iteration}
 F (T) \leq D_1 (T , \lambda_0) + \sum_{k=1}^N D_1 (4^{-k}T , \lambda_k) \prod_{l = 0}^{k-1} D_2 (4^{-l}T , \lambda_l) \\
 + F (4^{-N-1} T) \prod_{k=0}^N D_2 (4^{-k}T, \lambda_k) .
\end{multline}
We now choose the sequence $(\lambda_k)_{k\in\N_0}$ given by $\lambda_k = \nu \alpha^k$ with
\begin{equation*}
\alpha = \begin{cases} \alpha_0 \quad &\text{if } T\le T_0 ,
\\ 
\displaystyle \alpha_0 \left(\frac{T}{T_0}\right)^\frac{\gamma_3}{\gamma_2} &\text{if } T > T_0 ,
\end{cases}
\quad \text{and}\quad
\nu = 
\begin{cases} 
\displaystyle \nu_0 \left(\frac{T_0}{T}\right)^\frac{\gamma_3}{\gamma_2-\gamma_1} &\text{if}\ T\le T_0 ,
\\  \nu_0 &\text{if} \ T > T_0,
\end{cases}
\end{equation*}
where
\begin{equation*}
\alpha_0 := (2\cdot 4^{\gamma_3})^\frac{1}{\gamma_2-\gamma_1}, \quad 
\nu_0    := \max\left\{\left(\frac{2\ln (4K_1)}{\euler \ln (2) d_1} \right)^{\frac{1}{\gamma_1}}, 2\lambda^*\right\}, 
\quad 
T_0      := \left(\frac{2 d_1 \alpha_0^{\gamma_2}}{d_3\nu_0^{\gamma_2-\gamma_1}} \right)^\frac{1}{\gamma_3}.
\end{equation*}
With this notation we have (in both cases $T \leq T_0$ and $T > T_0$) the equality
\begin{equation} \label{eq:Prop_av}
d_3 T^{\gamma_3} \nu^{\gamma_2-\gamma_1} = 2 d_1 \alpha^{\gamma_2} ,
\end{equation}
which we will use frequently in the following. 
Moreover, the choice of $\alpha$ and $\nu$ ensures that the constants 
\begin{equation*}
K_2 := d_3 \Big(\frac{T}{4}\Big)^{\gamma_3} \nu^{\gamma_2} - d_1 \nu^{\gamma_1} \quad \text{and} \quad 
K_3 :=  \frac{K_2}{\alpha^{\gamma_2} / 4^{\gamma_3} - 1} - d_1 \nu^{\gamma_1} 
\end{equation*}
are positive. Indeed, using \eqref{eq:Prop_av} we find
\[
 K_3 = \frac{d_3 (T/4)^{\gamma_3} \nu^{\gamma_2} - \alpha^{\gamma_2} d_1 \nu^{\gamma_1} / 4^{\gamma_3}}{\alpha^{\gamma_2} / 4^{\gamma_3} - 1} 
 = \nu^{\gamma_1}\frac{d_1 \alpha^{\gamma_2} / 4^{\gamma_3}}{\alpha^{\gamma_2} / 4^{\gamma_3} - 1} .
\]
Since $\alpha^{\gamma_2} > 2\cdot 4^{\gamma_3}$, we conclude that $K_3$ is positive. Note that $K_2 > K_3$; hence $K_2$ is positive as well. 
Let us now show that the right-hand side in \eqref{eq:after_iteration} converges for $N \to \infty$. Since $\alpha^{\gamma_1} \le \alpha^{\gamma_2}/4^{\gamma_3}$, we have
\begin{align}
\prod_{k=0}^N D_2 (4^{-k}T, \lambda_k ) &= \prod_{k=0}^N K_1 \exp\left(d_1 \nu^{\gamma_1} \alpha^{{\gamma_1}k} - d_3 \left(\frac{T}{4}\right)^{\gamma_3} \nu^{\gamma_2} \left(\frac{\alpha^{\gamma_2}}{4^{\gamma_3}}\right)^k\right) \nonumber \\
&\le K_1^{N+1} \prod_{k=0}^N \exp\left(- K_2 \left(\frac{\alpha^{\gamma_2}}{4^{\gamma_3}}\right)^k\right). \label{eq:prod}
\end{align}
Since $K_1,K_2 > 0$ and $\alpha^{\gamma_2}/4^{\gamma_3} > 1$, this tends to zero as $N$ tends to infinity. Moreover, using \eqref{eq:prod} and $\alpha^{\gamma_1} \le \alpha^{\gamma_2}/4^{\gamma_3}$, we infer that the middle term of the right-hand side of \eqref{eq:after_iteration} satisfies
\begin{align*}
\sum_{k=1}^N & D_1(4^{-k}T, \lambda_k)  \prod_{\ell=0}^{k-1} D_2(4^{-\ell}T, \lambda_l) \\
&\le 2M\euler^{\omega_+ T} d_0 \frac{1}{T} \int_0^T G(\tau) \drm \tau \sum_{k=1}^N   (4K_1)^k \exp\left(-K_2 \frac{(\alpha^{\gamma_2} / 4^{\gamma_3} )^k - 1}{\alpha^{\gamma_2} / 4^{\gamma_3} - 1} + d_1\nu^{\gamma_1} \alpha^{{\gamma_1}k}\right) \\
&\le 2M\euler^{\omega_+ T} d_0 \frac{1}{T} \int_0^T G(\tau) \drm \tau  \exp\left(\frac{K_2}{\alpha^{\gamma_2} / 4^{\gamma_3} - 1}\right) \sum_{k=1}^N (4K_1)^k \exp\left(-K_3 \left(\frac{\alpha^{\gamma_2}}{4^{\gamma_3}}\right)^k\right).
\end{align*}
Since $K_3 > 0$, the right-hand side converges as $N$ tends to infinity, and we obtain from \eqref{eq:after_iteration} that
\begin{equation*}
\lVert S_Tx \rVert_X \le \tilde C_{\mathrm{obs}} \int_0^T \lVert CS_tx \rVert_Y \drm t,
\end{equation*}
where
\begin{equation*}
\tilde C_{\mathrm{obs}} 
= 
\frac{2M d_0}{T \euler^{-\omega_+ T}} \left( 
\euler^{d_1\nu^{\gamma_1}} + 
\exp\left(\frac{K_2}{\alpha^{\gamma_2} / 4^{\gamma_3} - 1}\right) \sum_{k=1}^\infty (4K_1)^k \exp\Bigg(-K_3 \left(\frac{\alpha^{\gamma_2}}{4^{\gamma_3}}\right)^k\Bigg) \right).
\end{equation*}
It remains to show the upper bound $\tilde C_{\mathrm{obs}} \leq C_{\mathrm{obs}}$ with $C_{\mathrm{obs}}$ as in the theorem. To this end, we note that for all $A>1$ and $B>0$ we have
\begin{equation*}
\sum_{k=1}^\infty A^k \euler^{-B2^k} \le \sup_{x\geq 1} A^x \euler^{-\frac{B}{2}2^x} \sum_{k=1}^\infty \euler^{-\frac{B}{2}2^k} = \left( \frac{2 \ln (A)}{B \euler \ln (2)} \right)^{\frac{\ln (A)}{\ln (2)}} \sum_{k=1}^\infty \euler^{-\frac{B}{2}2^k} ,
\end{equation*}
where the last identity follows from elementary calculus. Using $2^k \geq 2k$ and $\euler^B-1 \geq B$, we further estimate
\begin{equation*}
\sum_{k=1}^\infty \euler^{-\frac{B}{2}2^k} \le \sum_{k=1}^\infty \euler^{-kB} = \frac{\euler^{-B}}{1-\euler^{-B}} \le \frac{1}{B}.
\end{equation*}
Hence, we find
\begin{equation}
\sum_{k=1}^\infty A^k \euler^{-B2^k} \le \left( \frac{2 \ln (A)}{B \euler \ln (2)} \right)^{\frac{\ln (A)}{\ln (2)}}\frac{1}{B} .
\label{eq:sum_estimate}
\end{equation}
We now apply inequality (\ref{eq:sum_estimate}) with $A=4K_1$ and $B=K_3$ and obtain by using $\alpha^{\gamma_2} / 4^{\gamma_3} \geq 2$
\begin{equation*}
\tilde C_{\mathrm{obs}} 
\leq 
\frac{2M \euler^{\omega_+ T}d_0 }{T} \left(
\euler^{d_1\nu^{\gamma_1}} +  \exp\left(\frac{K_2}{\alpha^{\gamma_2} / 4^{\gamma_3} - 1}\right) 
\left( \frac{2 \ln (4K_1)}{K_3 \euler \ln (2)} \right)^{\frac{\ln (4K_1)}{\ln (2)}} \frac{1}{K_3} \right) .
\end{equation*}
For $K_3$ we have the lower bound
\begin{equation} \label{eq:K3}
 K_3 = \nu^{\gamma_1}\frac{d_1 \alpha^{\gamma_2} / 4^{\gamma_3}}{\alpha^{\gamma_2} / 4^{\gamma_3} - 1} 
 \geq 
 \nu^{\gamma_1} d_1 \geq \nu_0^{\gamma_1} d_1 \geq \frac{2 \ln (4K_1)}{\euler \ln (2)} .
\end{equation}
Since $M\euler^{\omega_+ T},d_2 \geq 1$ we have $K_1 \geq 1$; hence $\ln (4 K_1) \geq \ln 4$ and $K_3 > 1$. From this, inequality~\eqref{eq:K3}, and $d_1\nu^{\gamma_1} \le d_1 \nu^{\gamma_1} + K_3 = K_2 / (\alpha^{\gamma_2}/4^{\gamma_3} - 1)$, we conclude
 \begin{align} \label{eq:almostCobs}
\tilde C_{\mathrm{obs}}
&\leq 
\frac{2 M d_0}{T \euler^{-\omega_+ T}} \left(
 \euler^{d_1\nu^{\gamma_1}} + \exp\left(\frac{K_2}{\alpha^{\gamma_2} / 4^{\gamma_3} - 1}\right) \right)
 \leq 
\frac{4M d_0}{T\euler^{-\omega_+ T}} \exp\left(\frac{K_2}{\alpha^{\gamma_2} / 4^{\gamma_3} - 1}\right)  
 .
\end{align}
For the constant $K_2$ we calculate, using \eqref{eq:Prop_av} and $\alpha^{\gamma_2} / 4^{\gamma_3} - 1 \geq (1/2) \alpha^{\gamma_2} / 4^{\gamma_3}$,
\begin{align*}
\frac{K_2}{\alpha^{\gamma_2} / 4^{\gamma_3} - 1} & 
= \frac{d_3 (T / 4 )^{\gamma_3} \nu^{\gamma_2} - d_1 \nu^{\gamma_1}}{\alpha^{\gamma_2} / 4^{\gamma_3} - 1}
\leq \frac{d_3 (T/4)^{\gamma_3} \nu^{\gamma_2}}{\alpha^{\gamma_2} / 4^{\gamma_3} - 1} 
= \nu^{\gamma_1}\frac{2d_1 \alpha^{\gamma_2} / 4^{\gamma_3}}{\alpha^{\gamma_2} / 4^{\gamma_3} - 1}
\leq 4 d_1 \nu^{\gamma_1} .
\end{align*}
By our choice of $\nu$ we have 
\begin{align*}
\nu^{\gamma_1} = 
\begin{cases} 
\displaystyle \left( 2 d_1 \alpha_0^{\gamma_2} d_3^{-1} \right)^{{\gamma_1}/({\gamma_2}-{\gamma_1})} \left( \frac{1}{T} \right)^{{\gamma_1}{\gamma_3}/({\gamma_2}-{\gamma_1})} &\text{if} \ T\le T_0 , \\[1ex]
\displaystyle \max \left\{\frac{2 \ln (4K_1)}{\euler \ln (2)d_1}, \left(2\lambda^*\right)^{\gamma_1}\right\}  &\text{if } T > T_0 ,
 \end{cases}
\end{align*}
and hence
\begin{equation}\label{eq:K2}
\frac{K_2}{\alpha^{\gamma_2} / 4^{\gamma_3} - 1} \leq 
4 \left(\frac{2^{\gamma_1} \alpha_0^{{\gamma_1}{\gamma_2}} d_1^{\gamma_2} / d_3^{\gamma_1}}{T^{{\gamma_1}{\gamma_3}}}\right)^{\frac{1}{{\gamma_2}-{\gamma_1}}}
+
\max \left\{\frac{8 \ln (4K_1)}{\euler \ln (2)}, 4d_1\left(2\lambda^*\right)^{\gamma_1}\right\} .
\end{equation}
From inequalities~\eqref{eq:almostCobs} and \eqref{eq:K2} we conclude
 \begin{align*}
  \tilde C_{\mathrm{obs}} &\leq
  \frac{4M d_0 }{T\euler^{-\omega_+ T}} 
  \exp\left(4 \left(\frac{2^{\gamma_1} \alpha_0^{{\gamma_1}{\gamma_2}} d_1^{\gamma_2} / d_3^{\gamma_1}}{T^{{\gamma_1}{\gamma_3}}}\right)^{\frac{1}{{\gamma_2}-{\gamma_1}}} \right)  
  \max\left\{ (4K_1)^{8/(\euler \ln 2)} , \euler^{4d_1\left(2\lambda^*\right)^{\gamma_1}} \right\} .
 \end{align*}
Finally, we insert the values of $\alpha_0$ and $K_1$ and factor out $\euler^{10\omega_+ T / (\euler \ln 2)}$ from the maximum to obtain the assertion.
\end{proof}
\section{Sharp geometric conditions for observability of elliptic operators in \texorpdfstring{$L_p (\R^d)$}{Lp}}
\label{sec:L_p(Rd)}
In this section we consider the case where $X = L_p (\R^d)$, $Y = L_p (\thickset)$ with $1 < p < \infty$, $A_p$ is an elliptic operator in $L_p (\R^d)$ associated with a 
strongly elliptic polynomial in $\R^d$ of degree $m \geq 2$, $(S_t)_{t \geq 0}$ the $C_0$-semigroup on $L_p (\R^d)$ generated by $-A_p$, and $C = \1_\thickset$ is the restriction operator of a function in $L_p(\R^d)$ to some measurable subset $\thickset \subset \R^d$, i.e., $\1_\thickset \in \cL (L_p (\R^d) , L_p (\thickset))$ and $\1_\thickset f = f $ on $\thickset$. Let $T > 0$. Our goal is to show that the system
\begin{equation} \label{eq:system_L_p}
\begin{aligned}
  \dot{x}(t) & = -A_p x(t), \quad &t&\in (0,T],\quad x(0)  = x_0 \in L_p (\R^d), \\
  y(t)  &=  \1_\thickset x(t), \quad &t& \in [0,T] ,
  \end{aligned}
\end{equation}
satisfies a final state observability estimate in $L_r ((0,T);L_p (\R^d))$, $r\in[1,\infty]$ if and only if $\thickset$ is a so-called thick set; cf.\ Definition~\ref{Def:thick_set}. In particular, if $\thickset$ is a thick set, we conclude an observability estimate with an explicit dependence of $C_{\mathrm{obs}}$ on $T$, the order $m$ of the operator $A_p$, and the geometry of the set $\thickset$.
\par
We start by recalling the class of elliptic operators $A_p$ which we consider.
We denote by $\mathcal{S}(\R^d)$ the Schwartz space of rapidly decreasing functions, which is dense in $L_p(\R^d)$ for all $1< p <\infty$.
For $f\in \mathcal{S}(\R^d)$ let $\F f\from\R^d\to\C$ be the Fourier transform of $f$ defined by
\[\F f (\xi) := \frac{1}{(2\pi)^{d/2}}\int_{\R^d} f(x) \euler^{-\ii\xi\cdot x}\,\drm x.\]
Then $\F\from \mathcal{S}(\R^d)\to \mathcal{S}(\R^d)$ is bijective and continuous and has a continuous inverse, given by
\[\F^{-1} f(x) = \frac{1}{(2\pi)^{d/2}} \int_{\R^d} f(\xi) \euler^{\ii x\cdot \xi}\,\drm \xi\]
for all $f\in \mathcal{S}(\R^d)$. 
Let $a \from \R^d \to \C$ be a homogeneous strongly elliptic polynomial of degree $m \geq 2$, that is, $a$ is of the form
\[
 a (\xi) = \sum_{\lvert \alpha \rvert_1 = m} a_\alpha \ii^{\lvert \alpha \rvert_1} \xi^\alpha
\]
for given $a_\alpha \in \C$, and there is $c > 0$ such that for all $\xi \in \R^d$ we have
\begin{equation*}
 \re a (\xi) \geq c \lvert \xi \rvert^m .
\end{equation*}
Note that this implies that $m$ is even. For $f\in \mathcal{S}(\R^d)$ define $A f \in \mathcal{S}(\R^d)$ by
\[
Af := \sum_{\lvert \alpha \rvert_1 =m} a_\alpha \partial^\alpha f = \F^{-1} (a \F f).
\]
Then, for every $1< p < \infty$, $A$ is closable in $L_p(\R^d)$, and its closure $A_p$ is a sectorial operator of angle $\omega_a < \pi / 2$. As a consequence, $-A_p$ generates a bounded $C_0$-semigroup $(S_t)_{t \geq 0}$ on $L_p (\R^d)$. We call $A_p$ the \emph{elliptic operator associated with $a$}. For details we refer, e.g., to the book \cite{Haase-06}. 
\begin{example}
  Let $1 < p <\infty$ and $a \colon \R^d \to \R$ defined by $a(\xi) = \lvert \xi \rvert^2$. Then $a$ is a homogeneous strongly elliptic polynomial of degree $m=2$ and $A_p = -\Delta$ is the negative Laplacian in $L_p(\R^d)$.
 \par
  More generally, let $ (a_{i,j}) \in\R^{d\times d}$ be a symmetric and negative definite matrix, and define $a \colon \R^d \to \R$ by $a(\xi) = \xi^{\top} (a_{i,j}) \xi$ for all $\xi\in \R^d$.
  Then $a$ is a homogeneous strongly elliptic polynomial of degree $m=2$ and $A_p = -\diverg\, (a_{i,j}) \grad$ is the corresponding elliptic operator in $L_p(\R^d)$.  
\end{example}
\par
The following definition characterizes the class of subsets $\thickset \subset \R^d$ which we consider.
\begin{definition}\label{Def:thick_set}
  Let $\rho\in (0,1]$ and $L\in (0,\infty)^d$. A set $\thickset \subset \R^d$ is called \emph{$(\rho,L)$-thick} if $\thickset$ is measurable and for all $x \in \R^d$ we have
  \[
  \left\lvert \thickset \cap \left( \bigtimes_{i=1}^d (0,L_i) + x \right) \right\rvert \geq \rho \prod_{i=1}^d L_i .
  \]
  Here, $\lvert \cdot \rvert$ denotes Lebesgue measure in $\R^d$.
  Moreover, $\thickset \subset \R^d$ is called \emph{thick} if there are $\rho\in (0,1)$ and $L\in (0,\infty)^d$ such that $\thickset$ is $(\rho,L)$-thick.  
\end{definition}

We are now in position to state our main theorems of this section.
\begin{theorem}
\label{Thm:Observability_Elliptic_Operator_Rd}
  Let $1<p<\infty$, $r\in[1,\infty]$, $a\colon\R^d \to \C$ a homogeneous strongly elliptic polynomial in $\R^d$ of degree $m\geq 2$, $A_p$ the associated elliptic operator in $L_p(\R^d)$, $(S_t)_{t\geq 0}$ the bounded $C_0$-semigroup on $L_p (\R^d)$ generated by $-A_p$,
  $\thickset \subset \R^d$ a $(\rho,L)$-thick set, and $T > 0$. Then the system~\eqref{eq:system_L_p} satisfies a final state observability estimate in $L_r ((0,T);L_p (\R^d))$. In particular, we have for all $x_0 \in L_p (\R^d)$
  \begin{align*}
  \lVert S_T x_0 \rVert_{L_p(\R^d)} \leq C_{\mathrm{obs}} \lVert \1_\thickset S_{(\cdot)} x_0 \rVert_{L_r ((0,T);L_p(\thickset))}
  \end{align*}
  with
	  \begin{align*}
C_{\mathrm{obs}}
& = \frac{D_1 M^{16}}{T^{1/r}} \left( \frac{K^d}{\rho} \right)^{D_2} \exp \left(\frac{D_3 (\lvert L \rvert_1 \ln (K^d / \rho))^{m/(m-1)}}{(c T)^{\frac{1}{m-1}}} \right),
 \end{align*}
where $K\geq 1$ is a universal constant, $D_1,D_2 \geq 1$ depending on $d$, $D_3 \geq 1$ depending on $d$ and $p$, $M = \sup_{t \geq 0} \lVert S_t \rVert$, and $c > 0$ is such that $\re a (\xi) \geq c \lvert \xi \rvert^m$ for all $\xi \in \R^d$.
\end{theorem}

The universal constant $K$ in Theorem~\ref{Thm:Observability_Elliptic_Operator_Rd} can be chosen to be the same as the constant $K$ in the Logvinenko--Sereda theorem (Theorem~\ref{Thm:Logvinenko-Sereda_Rd}). Theorem~\ref{Thm:Observability_Elliptic_Operator_Rd} shows that the system \eqref{eq:system_L_p} satisfies a final state observability estimate if $\thickset$ is a thick set. Note that $C_{\mathrm{obs}}$ is optimal in $T$ (see Remark \ref{rem:UR-OBS}\ref{rem:UR-OBS:2}), as well as in the geometric parameters $\rho$ and $L$ by \cite[Remark 4.14]{NakicTTV-18b-arxiv}.
The following theorem shows the converse: If the system \eqref{eq:system_L_p} satisfies a final state observability estimate, then the set $\thickset$ is necessarily a thick set.
\begin{theorem}
\label{Thm:Thick_set_Rd}
  Let $1<p<\infty$, $r\in[1,\infty]$, $a\colon\R^d \to \C$ a homogeneous strongly elliptic polynomial in $\R^d$ of degree $m\geq 2$, $A_p$ the associated elliptic operator in $L_p(\R^d)$, $(S_t)_{t\geq 0}$ the bounded $C_0$-semigroup on $L_p (\R^d)$ generated by $-A_p$,
  $\thickset\subset \R^d$ measurable, and $T > 0$, and assume there exists $C_{\mathrm{obs}}>0$ such that for all $x_0 \in L_p (\R^d)$ we have
  \[
  \lVert S_T x_0 \rVert_{L_p(\R^d)} \leq C_{\mathrm{obs}} \lVert \1_\thickset S_{(\cdot)} x_0 \rVert_{L_r ((0,T);L_p(\thickset) )}.
  \]
  Then $\thickset$ is a thick set.
\end{theorem}

In order to prove Theorem~\ref{Thm:Observability_Elliptic_Operator_Rd} we apply Theorem~\ref{thm:spectral+diss-obs} in the case where $X = L_p (\R^d)$, $Y = L_p(\thickset)$, $C\in \mathcal{L}(X,Y)$ is the restriction operator of functions from $\R^d$ to $\thickset$, and $A = A_p$. To this end we define a family $(P_\lambda)_{\lambda > 0}$ of operators in $L_p (\R^d)$ such that the assumptions of Theorem~\ref{thm:spectral+diss-obs}, i.e., the uncertainty relation \eqref{eq:ass:uncertainty} and the dissipation estimate \eqref{eq:ass:dissipation}, are satisfied. Concerning the dissipation estimate we first consider the case $p = 2$ and then apply the Riesz--Thorin interpolation theorem. For the uncertainty relation we shall need a so-called Logvinenko--Sereda theorem. 
It was originally proven by Logvinenko and Sereda in \cite{LogvinenkoS-74} and significantly improved by Kovrijkine in \cite{Kovrijkine-00,Kovrijkine-01}. Recently, it has been adapted to functions on the torus instead of $\R^d$; see \cite{EgidiV-16-arxiv}. We quote a special case of Theorem~1 from \cite{Kovrijkine-01}.
\begin{theorem}[Logvinenko--Sereda theorem]
\label{Thm:Logvinenko-Sereda_Rd}
  There exists $K\geq 1$ such that for all $p\in [1,\infty]$, all $\lambda > 0$, all $\rho \in (0,1]$, all $L \in (0,\infty)^d$, all $(\rho,L)$-thick sets $\thickset \subset \R^d$, and all $f\in \mathcal{S}(\R^d)$ satisfying $\supp \F f \subset [-\lambda,\lambda]^d$ we have
  \[
  \lVert f \rVert_{L_p(\R^d)} \leq \Bigl(\frac{\rho}{K^d}\Bigr)^{-K(d+2 \lambda \lvert L \rvert_1)} \lVert f \rVert_{L_p(\thickset)}.
  \]
\end{theorem}

We now proceed with the proofs of Theorems~\ref{Thm:Observability_Elliptic_Operator_Rd} and \ref{Thm:Thick_set_Rd}.
\begin{proof}[Proof of Theorem~\ref{Thm:Observability_Elliptic_Operator_Rd}]
We apply Theorem~\ref{thm:spectral+diss-obs} in the case where $X = L_p (\R^d)$, $Y=L_p(\thickset)$, $C\in \mathcal{L}(X,Y)$ is the restriction of functions from $\R^d$ to $\thickset$, and $A = A_p$. For this purpose we define a family $(P_\lambda)_{\lambda > 0}$ of operators in $L_p (\R^d)$ such that the assumptions of Theorem~\ref{thm:spectral+diss-obs} are satisfied. 
Let $\eta\in C_{\mathrm c}^\infty ([0,\infty) )$ with $0\leq\eta\leq 1$ such that $\eta (r) = 1$ for $r\in [0,1/2]$ and $\eta (r) = 0$ for $r\geq 1$. 
  For $\lambda > 0$ we define $\chi_\lambda\from \R^d\to \R$ by $\chi_\lambda (\xi) = \eta (\lvert \xi \rvert / \lambda)$. Since $\chi_\lambda \in \mathcal{S}(\R^d)$ for all $\lambda >0$, we have $\mathcal{F}^{-1}\chi_\lambda \in \mathcal{S}(\R^d)$. For $\lambda > 0$ we define $P_\lambda \from  L_p(\R^d) \to L_p(\R^d)$ by $P_\lambda f = (2\pi)^{-d/2}(\mathcal{F}^{-1} \chi_\lambda) \ast f$.
By Young's inequality we have for all $f\in L_p(\R^d)$
\[
 \lVert P_\lambda f \rVert_{L_p(\R^d)} = \lVert (2\pi)^{-d/2}(\mathcal{F}^{-1} \chi_\lambda) \ast f \rVert_{L_p(\R^d)} \leq (2\pi)^{-d/2}\lVert \mathcal{F}^{-1} \chi_\lambda \rVert_{L_1(\R^d)} \lVert f \rVert_{L_p(\R^d)}.
\]
Moreover, the norm $\lVert \mathcal{F}^{-1} \chi_\lambda \rVert_{L_1(\R^d)}$ is independent of $\lambda >0$. Indeed, by the scaling property of the Fourier transform and by change of variables we have for all $\lambda > 0$
\begin{align*}
\lVert \mathcal{F}^{-1} \chi_\lambda \rVert_{L_1(\R^d)} 
= 
\lvert \lambda \rvert^d \lVert (\mathcal{F}^{-1} \chi_1) (\lambda \cdot)\rVert_{L_1(\R^d)}
= 
\lVert \mathcal{F}^{-1} \chi_1 \rVert_{L_1(\R^d)}.
\end{align*}
Hence, for all $\lambda > 0$ the operator $P_\lambda$ is a bounded linear operator, and the family $(P_\lambda)_{\lambda>0}$ is uniformly bounded by $C_d:=(2\pi)^{-d/2}\lVert \mathcal{F}^{-1} \chi_1 \rVert_{L_1(\R^d)}$.
For all  $f\in \mathcal{S}(\R^d)$ we have by construction $P_\lambda f\in\mathcal{S}(\R^d)$, $\F P_\lambda f = \chi_\lambda \F f \in \mathcal{S}(\R^d)$, and $\supp \F P_\lambda f \subset \{y \in \R^d \colon \lvert y \rvert \leq \lambda\} \subset [-\lambda, \lambda]^d$.
  By Theorem~\ref{Thm:Logvinenko-Sereda_Rd} we obtain for all $\lambda > 0$ and all $f \in \mathcal{S}(\R^d)$
  \begin{subequations}\label{eq:uncertainty_L_p}
  \begin{equation}
  \lVert P_\lambda f \rVert_{L_p(\R^d)} \leq d_0 \euler^{d_1 \lambda}  \lVert P_\lambda f \rVert_{L_p(\thickset)} ,
  \end{equation}
  where
  \begin{equation} 
   d_0 = \euler^{-Kd \ln (\rho / K^d)}
   \quad\text{and}\quad
   d_1 = - 2 K \lvert L \rvert_1 \ln \left(\frac{\rho}{K^d} \right) .
  \end{equation}
  \end{subequations}
  Since $\mathcal{S} (\R^d)$ is dense in $L_p (\R^d)$ and $P_\lambda$ is bounded, inequality~\eqref{eq:uncertainty_L_p} holds for all $f \in L_p (\R^d)$. Thus, the uncertainty relation \eqref{eq:ass:uncertainty} of Theorem~\ref{thm:spectral+diss-obs} is satisfied with $d_0$ and $d_1$ as in \eqref{eq:uncertainty_L_p}, $\gamma_1 = 1$, and $\lambda^* = 0$.
  \par
  It remains to verify the dissipation estimate. Since $\F S_t \F^{-1} f = \euler^{-ta} f$ for $f \in \mathcal{S} (\R^d)$ by functional calculus arguments (see, e.g., section~8 in \cite{Haase-06}), and since the Fourier transform is an isometry in $L_2 (\R^d)$, we obtain for all $f\in\mathcal{S}(\R^d)$ and all $\lambda > 0$
  \begin{align*}
  \lVert (1-P_\lambda)S_t f \rVert_{L_2(\R^d)}
  &= 
  \lVert \F^{-1} (1-\chi_\lambda) \euler^{-ta} \F f\rVert_{L_2(\R^d)} \\ &\leq 
  \lVert (1-\chi_\lambda)\euler^{-ta}\rVert_{L_\infty(\R^d)} \lVert f\rVert_{L_2(\R^d)}.
    \end{align*}
  Since 
	$1-\chi_\lambda \le 1-\1_{\ball{0}{\lambda / 2}} \leq 1$, $\re a (\xi) \geq c \lvert \xi \rvert^m$ for all $\xi \in \R^d$, and since $\mathcal{S} (\R^d)$ is dense in $L_2 (\R^d)$, this yields for all $f \in L_2 (\R^d)$
  \begin{align}\label{eq:dissipation_L2}
   \lVert (1-P_\lambda)S_t f \rVert_{L_2(\R^d)} & 
    \leq \euler^{-ct (\lambda / 2)^m} \lVert f \rVert_{L_2(\R^d)}.
  \end{align}
  This shows that the dissipation estimate \eqref{eq:ass:dissipation} of Theorem~\ref{thm:spectral+diss-obs} is satisfied if $p = 2$. In order to treat the case $p \not = 2$ we apply the Riesz--Thorin interpolation theorem. 
  Let 
  \[
   p_0 := \begin{cases}
          p^2-2p+2 & \text{if $p \in (1,2)$,} \\
          2 p & \text{if $p \in (2,\infty)$,} 
         \end{cases}
   \quad\text{and}\quad
   \theta := \begin{cases}
             \frac{-2p^2+6p-4}{-p^3+2p^2} & \text{if $p \in (1,2)$,} \\
             \frac{1}{p-1} &  \text{if $p \in (2,\infty)$.} 
             \end{cases}
  \]
 If $p = 2$ we set $p_0 := 2$ and $\theta := 1$ for convenience.
 Then $p_0\in (1,\infty)$, $\theta\in (0,1]$, and
 \[
 \frac{1}{p} = \frac{1-\theta}{p_0} + \frac{\theta}{2}.
 \]
Since the family $(P_\lambda)_{\lambda>0}$ is uniformly bounded by $C_d$, we have for all $f\in L_{p_0}(\R^d)$ and all $\lambda > 0$
 \begin{align*}
 \lVert (1-P_\lambda) S_t f \rVert_{L_{p_0}(\R^d)} 
 \leq \bigl(1+ C_d \bigr) M \lVert f \rVert_{L_{p_0}(\R^d)} ,
 \end{align*}
where $M = \sup_{t \geq 0} \lVert S_t \rVert$.
Interpolation between $L_2(\R^d)$ and $L_{p_0}(\R^d)$ if $p \not = 2$, and inequality~\eqref{eq:dissipation_L2} if $p = 2$, now yields for all $f \in L_p (\R^d)$
\begin{subequations}\label{eq:dissipation_L_p}
\begin{equation} 
\lVert (1-P_\lambda)S_t f \rVert_{L_p(\R^d)} \leq d_2 \euler^{-d_3 t \lambda^m} \lVert f \rVert_{L_p(\R^d)} ,
\end{equation}
where
\begin{equation}
 d_2 = (1+C_d)^{1-\theta} M ^{1-\theta} \quad\text{and} \quad
 d_3 = c\theta / 2^m .
\end{equation}
\end{subequations}
Thus, the dissipation estimate \eqref{eq:ass:uncertainty} of Theorem~\ref{thm:spectral+diss-obs} is satisfied with $d_2$ and $d_3$ as in \eqref{eq:dissipation_L_p}, $\gamma_2 = m$, $\gamma_3 = 1$, and $\lambda^* = 0$.
Since $\gamma_2 = m > 1 = \gamma_1$, we conclude from the uncertainty relation \eqref{eq:uncertainty_L_p}, the dissipation estimate \eqref{eq:dissipation_L_p}, and Theorem~\ref{thm:spectral+diss-obs} that the statement of the theorem holds with
\[
\tilde C_{\mathrm{obs}} = \frac{4Md_0 (4K_1)^{\frac{8}{\euler \ln 2}}}{T^{1/r}} \exp \Biggl(\frac{4 \bigl(2 \cdot 8^\frac{m}{m-1} d_1^{m} / d_3 \bigr)^{\frac{1}{m-1}}}{T^{\frac{1}{m - 1}}}\Biggr) ,
\]
where $K_1 = (d_0 +1)M^{2}d_2$, and $T^{1/r} = 1$ if $r = \infty$. From the definitions of $d_i$, $i \in \{0,1,2,3\}$ and straightforward estimates we obtain $\tilde C_{\mathrm{obs}} \leq C_{\mathrm{obs}}$ with $C_{\mathrm{obs}}$ as in the theorem.
\end{proof}
\begin{remark}
  As the proof shows we can obtain an explicit dependence of $C_{\mathrm{obs}}$ on $p$. Then, it turns out that $C_{\mathrm{obs}} \to \infty$ as $p\to 1$ and as $p\to\infty$.
  This shows that our method of proof is only valid for $p\in (1,\infty)$.
\end{remark}
\begin{proof}[Proof of Theorem~\ref{Thm:Thick_set_Rd}]
  We improve the strategy developed in \cite{EgidiV-18}. 
  We show the contraposition. Assume that $\thickset$ is not thick. Then there exists a sequence $(x_n)_{n \in \N}$ in $\R^d$ such that for all $n \in \N$ we have
  \begin{equation}\label{eq:not-thick}
  \lvert \thickset\cap \ball{x_n}{n} \rvert < \frac{1}{n} .
  \end{equation}
  Note that $\euler^{-ta}\in\mathcal{S} (\R^d)$ for $t>0$. Thus, for $t>0$, the operator $S_t$ is given as a convolution operator with convolution kernel
  $p_t :=(1 / (2\pi)^{d/2}) \F^{-1}\euler^{-ta} \in \mathcal{S} (\R^d)$ for $t>0$. Indeed, for $f\in \mathcal{S}(\R^d)$ and $t>0$ we have $\euler^{-ta} \F f\in \mathcal{S}(\R^d)$ and
  \[S_t f = \F^{-1} \euler^{-ta} \F f = (2\pi)^{d/2} \F^{-1} (\F p_t \F f) = p_t * f,\]
  and the claim follows by density.
  For $n\in\N$ we define $f_n := p_1(\cdot-x_n)$. As a consequence, we observe 
  for all $t>0$ and $n\in\N$
  \begin{equation}\label{eq:Faltung}
  S_t f_n = p_t * f_n = p_t * p_1(\cdot-x_n) = p_{t+1}(\cdot-x_n),
  \end{equation}
  and hence by translation invariance of the Lebesgue measure
  \begin{equation} \label{eq:1a}
  \lVert S_T f_n \rVert_{L_p(\R^d)} = \lVert p_{T+1}(\cdot-x_n) \rVert_{L_p(\R^d)} = \lVert p_{T+1} \rVert_{L_p(\R^d)}.
  \end{equation}
 For $n \in \N$ we now shift the set $\thickset$ by $x_n$ and consider the set $\thickset - x_n = \{y \in \R^d \colon y + x_n \in \thickset\}$. Note that \eqref{eq:not-thick} is equivalent to $\lvert (\thickset-x_n) \cap \ball{0}{n} \rvert < 1/n$ for all $n \in \N$. From the latter fact, \eqref{eq:Faltung}, and substitution we obtain for all $t > 0$ and $n \in \N$
 \begin{align}
    \lVert \1_\thickset S_t f_n \rVert_{L_p(\R^d)}^p 
    &= 
    \lVert \1_{(\thickset-x_n) \cap \ball{0}{n}} p_{t+1} \rVert_{L_p(\R^d)}^p + 
    \lVert \1_{(\thickset-x_n)} (1-\1_{\ball{0}{n}}) p_{t+1} \rVert_{L_p(\R^d)}^p \nonumber \\[1ex]
    &\leq
    \lVert p_{t+1} \rVert_{L_{\infty} (\R^d)}^p \lvert (\thickset-x_n) \cap \ball{0}{n} \rvert +
    \lVert (1-\1_{\ball{0}{n}}) p_{t+1} \rVert_{L_p(\R^d)}^p \nonumber \\
    & < 
    \frac{1}{n}\lVert p_{t+1} \rVert_{L_{\infty}(\R^d)}^p  +
    \lVert  (1-\1_{\ball{0}{n}}) p_{t+1} \rVert_{L_p(\R^d)}^p . 
    \label{eq:split-integrals}
 \end{align}
 Since $a$ is a homogeneous polynomial, we have by substitution $p_t (x) = t^{-d/m} p_1 (x \allowbreak{} / t^{1/m})$. Hence, we find for all $t > 0$
 \begin{equation} \label{eq:eq1}
  \lVert p_{t+1} \rVert_{L_{\infty}(\R^d)} = \frac{1}{(t+1)^{d/m}} \lVert p_{1} \rVert_{L_{\infty}(\R^d)} \leq \lVert p_{1} \rVert_{L_{\infty}(\R^d)} .
 \end{equation}
 Moreover, it follows for all $t \in (0,T]$ that
 \begin{align}
  \lVert  (1-\1_{\ball{0}{n}}) p_{t+1} \rVert_{L_p(\R^d)}^p 
  & =
  \int_{\R^d}  \bigl(1-\1_{\ball{0}{n}}(x)\bigr) \biggl\lvert \frac{p_1 (x/(t+1)^{1/m})}{(t+1)^{d/m}}   \biggr\rvert^p \drm x \nonumber \\[1ex]
  & = 
  \int_{\R^d}  \bigl(1-\1_{\ball{0}{n / (t+1)^{1/m}}}(x)\bigr) \bigl\lvert p_1 (x)  \bigr\rvert^p \frac{1}{(t+1)^{(p-1)d/m}}\drm x \nonumber \\
  & \leq
  \int_{\R^d}  \bigl(1-\1_{\ball{0}{n/ (T+1)^{1/m}}}(x)\bigr) \bigl\lvert p_1 (x)  \bigr\rvert^p \drm x .
  \label{eq:eq2}
 \end{align}
 From \eqref{eq:split-integrals}, \eqref{eq:eq1}, and \eqref{eq:eq2} (and since $p_1$ is a Schwartz function and hence integrable), we obtain that
 \begin{equation} \label{eq:2a}
  \lVert \1_\thickset S_{(\cdot)} f_n \rVert_{L_r ((0,T) ; L_p (\R^d))} \to 0
 \end{equation}
 as $n$ tends to infinity.
  From \eqref{eq:1a} and \eqref{eq:2a} we conclude that for all $C_{\mathrm{obs}} > 0$ there exists $x_0 \in L_p (\R^d)$ such that
  \[
   \lVert S_T x_0 \rVert_{L_p (\R^d)} > C_{\mathrm{obs}} \lVert \1_{\thickset} S_{(\cdot)} x_0 \rVert_{L_r((0,T);L_p (\thickset))}.
  \]
  This proves the contraposition of the theorem.
\end{proof}
\section{Null-controllability and control costs}
\label{sec:Control_Costs}
Let $X$ and $U$ be Banach spaces, $(S_t)_{t \geq 0}$ be a $C_0$-semigroup on $X$, $-A$ the corresponding infinitesimal generator on $X$, $B \in \cL (U,X)$, and $T > 0$. We consider the linear control system
\begin{align}
  \dot x(t) & = -Ax(t) + Bu(t), \quad t \in (0,T], \quad x(0) = x_0 \in X,
	\label{eq:System}
\end{align}
where $u\in L_r ((0,T);U )$ with $1 \le r \le \infty$. 
The function $x$ is called \emph{state function}, and $u$ is called \emph{control function}.
The unique mild solution of \eqref{eq:System} is given by Duhamel's formula
\begin{equation*}
x(t) = S_t x_0 + \int_0^t S_{t-\tau} Bu(\tau) \drm\tau , \quad t \in [0,T].
\end{equation*}
We say that the system \eqref{eq:System} is \emph{null-controllable in time $T$ via $L_r ((0,T);U )$} if for all $x_0\in X$ there exists $u\in L_r ((0,T);U )$ such that $x (T) = 0$.
The \emph{controllability map} is given by $\mathcal{B}^T \colon L_r ((0,T); U) \to X$, 
\begin{equation}\label{eq:controllability-map}
\mathcal{B}^T u := \int_0^T S_{T-\tau} Bu(\tau) \drm\tau .
\end{equation}
Note that we suppress the dependence of $\mathcal{B}^T$ on $r$. The system \eqref{eq:System} is null-controllable in time $T$ via $L_r ((0,T);U )$ if and only if $\ran \mathcal{B}^T \supset \ran S_T$. This gives an alternative definition of null-controllability. 
\par
Denote by $A'$ in $X'$ the dual operator of $A$ and by $B' \in \cL (X' , U')$ the dual operator of $B$. It is well known that null-controllability of the system \eqref{eq:System} is in certain situations equivalent to final state observability of its adjoint or dual system 
\begin{equation} \label{eq:AdjointSystem}
\begin{aligned}
  \dot{\varphi}(t) & = -A' \varphi(t), \quad & & t\in (0,T]  ,\quad \varphi(0)  = \varphi_0 \in X', \\
  \psi(t)  &= B' \varphi(t), \quad   & & t\in [0,T]  .
  \end{aligned}
\end{equation}
Recall that the system~\eqref{eq:AdjointSystem} satisfies a final state observability estimate in $L_{r'} ((0,T);\allowbreak{} U')$, $r' \in [1,\infty]$ if there exists $C_{\mathrm{obs}} > 0$ such that for all $\varphi_0 \in X'$ we have $\lVert \varphi(T) \rVert_{X'} \leq C_{\mathrm{obs}} \lVert \psi \rVert_{L_{r'} ((0,T); \allowbreak{} U')}$.
This equivalence can be described in an abstract form due to Douglas \cite{Douglas-66} and Dolecki and Russell \cite{DoleckiR-77}; see in particular Theorem~2.5 and Section~5 in \cite{DoleckiR-77}.
\begin{lemma}[{\cite{Douglas-66,DoleckiR-77}}]\label{lemma:DR}
Let $V,W,Z$ be reflexive Banach spaces, and let $F\in \cL(V,Z)$, $G\in \cL(W,Z)$. Then the following are equivalent:
\begin{enumerate}[(a)]
	\item $\ran F \subset \ran G$.
	\item There exists $c_1>0$ such that $\lVert F'z' \rVert_{V'} \leq c_1 \lVert G'z' \rVert_{W'}$ for all $z'\in Z'$.
	\item There exists $H\colon \overline{\ran G'} \to V'$ and $c_2>0$ such that $HG' = F'$ and $\lVert Hw' \rVert_{V'} \le c_2 \lVert w' \rVert_{W'}$ for all $w' \in \overline{\ran G'}$.
\end{enumerate}
Moreover, in (b) and (c) we can choose $c_1=c_2$.
\end{lemma}
In particular, if $V = Z = X$, $W=L_r ((0,T); U )$, $F=S_T$ and $G=\mathcal{B}^T$, statement (a) of Lemma~\ref{lemma:DR} is equivalent to the fact that the system~\eqref{eq:System} is null-controllable in time $T>0$ via $L_r ((0,T);U )$, while statement (b) of Lemma~\ref{lemma:DR} is equivalent to the fact that the system~\eqref{eq:AdjointSystem} satisfies a final state observability estimate in $L_{r'} ((0,T) ; U')$, where $1/r + 1/r' = 1$ provided $r \in (1,\infty)$. Thus, if $X$ and $U$ are reflexive and $r\in (1,\infty)$, Lemma~\ref{lemma:DR} implies that null-controllability of the system \eqref{eq:System} is equivalent to final state observability of the adjoint or dual system \eqref{eq:AdjointSystem}.
More recently, in \cite{Vieru-05} and \cite{YuLC-06} it is shown that this equivalence holds true even if $X$ is a general Banach space, $U$ a reflexive Banach space, and $r\in (1,\infty]$. 

\begin{theorem}[{\cite{Vieru-05,YuLC-06}}]
\label{Thm:Duality}
Let $X$ and $U$ be Banach spaces, $U$ reflexive, $(S_t)_{t\geq 0}$ a  $C_0$-semigroup on $X$, $-A$ the corresponding infinitesimal generator on $X$, $B\in \cL(U,X)$, $r\in (1,\infty]$, and $r' \in [1,\infty)$ with $1 / r + 1 / r' = 1$. Then the system \eqref{eq:System} is null-controllable in time $T>0$ via $L_r ((0,T);U )$ if and only if there exists $C_\mathrm{obs}>0$ such that 
\[
\forall x' \in X' \colon \quad \lVert S'_T x' \rVert_{X'} 
\leq 
C_\mathrm{obs} \lVert B' S_{(\cdot)}' x'\rVert_{L_{r'}((0,T);U')}.\]
\end{theorem}

Note that $-A'$ in general does not generate a $C_0$-semigroup on $X'$ but is merely the weak$^*$ generator of the weak$^*$-continuous semigroup $(S_t')_{t \geq 0}$ on $X'$ given by $S_t' := (S_t)'$ for all $t\geq 0$.
However, if $X$ is reflexive, then $(S_t')_{t \geq 0}$ is strongly continuous and $-A'$ is the infinitesimal generator of $(S_t')_{t \geq 0}$. If we assume that $(S_t')_{t \geq 0}$ is strongly continuous, we can combine Theorem~\ref{thm:spectral+diss-obs} and Theorem~\ref{Thm:Duality} and obtain sufficient conditions for null-controllability of the system~\eqref{eq:System}.
\begin{theorem} Let $X, U$ be Banach spaces, $U$ reflexive, $(S_t)_{t\geq 0}$ a $C_0$-semigroup on $X$, $-A$ the corresponding infinitesimal generator on $X$, $B\in \cL(U,X)$, and assume that $(S'_t)_{t\geq0}$ is strongly continuous. 
Let further $\lambda^*\geq 0$ and $(P'_\lambda)_{\lambda>\lambda^*}$ be a family of bounded linear operators in $X'$, $r \in (1,\infty]$, $d_0,d_1,d_3,\gamma_1,\gamma_2,\gamma_3,T > 0$ with $\gamma_1 < \gamma_2$, $d_2\geq 1$, and assume that
\begin{equation*}
\forall x'\in X' \ \forall \lambda > \lambda^* \colon \quad \lVert P'_\lambda x' \rVert_{ X' } \le d_0 {\euler}^{d_1 \lambda^{\gamma_1}} \lVert B'  P'_\lambda x' \rVert_{U' }
\end{equation*}
and 
\begin{equation*}
\forall x'\in X' \ \forall \lambda > \lambda^* \ \forall t\in (0,T/2] \colon \quad \lVert (\id-P'_\lambda) S'_t x' \rVert_{X'} \le d_2 {\euler}^{-d_3 \lambda^{\gamma_2} t^{\gamma_3}} \lVert x' \rVert_{X'} .
\end{equation*}
Then the system \eqref{eq:System} is null-controllable in time $T$ via $L_r ((0,T);U)$.
\end{theorem}

Combining Theorem~\ref{Thm:Duality} with Theorems~\ref{Thm:Observability_Elliptic_Operator_Rd} and \ref{Thm:Thick_set_Rd} we obtain a sharp geometric condition on null-controllability for linear systems governed by strongly elliptic operators with interior control.
For $\thickset \subset \R^d$ measurable we denote by $\1_\thickset' \in \cL (L_p (\thickset) , L_p (\R^d))$ the canonical embedding, i.e., $\1_\thickset' f  = f$ on $\thickset$ and $\1_\thickset' f = 0$ on $\R^d\setminus\thickset$. 
\begin{theorem} \label{theorem:null<->thick}
  Let $1<p<\infty$, $a \colon \R^d \to \C$ a homogeneous strongly elliptic polynomial in $\R^d$ of degree $m\geq 2$, $A_p$ the associated elliptic operator in $L_p(\R^d)$, $\thickset\subset \R^d$ measurable, $r\in (1,\infty]$, and $T>0$.
  Then the system 
	\begin{align*}
	\dot{x}(t) & = -A_p x(t) + \1_\thickset' u(t), \quad t\in (0,T], \quad x(0) = x_0 \in L_p(\R^d)
\end{align*}
is null-controllable in time $T$ via $L_{r} ((0,T);L_p(\R^d))$ if and only if $\thickset$ is a thick set.
\end{theorem}
\begin{proof}
Let $1<p'<\infty$ be such that $1 / p + 1 / p' = 1$. Note that $L_p(\R^d)' \cong L_{p'}(\R^d)$, $(A_p)' = A_{p'}$ (note that $m$ is even), and $A_{p'}$ is the generator of the bounded $C_0$-semigroup $(S_t')_{t\geq 0}$.
If we set $B = \1_\thickset' \in \cL (L_p(\thickset),L_p(\R^d))$ we have for the dual operator $B' = \1_\thickset \in \cL (L_{p'}(\R^d),L_{p'}(\thickset) )$, i.e., $B'$ is the restriction operator of a function $x\in L_{p'}(\R^d)$ on $\thickset$.
Hence, combining Theorem~\ref{Thm:Duality} with Theorems~\ref{Thm:Observability_Elliptic_Operator_Rd} and \ref{Thm:Thick_set_Rd} for the dual system we obtain the assertion.
\end{proof}

We now turn to the discussion of the control costs. For $T>0$ we call the quantity 
\begin{equation*}
C_T := \sup_{\stackrel{x_0\in X}{\lVert x_0 \rVert_X=1}} \inf \left\{ \lVert u \rVert_{L_r((0,T);U)} \colon u\in L_r ((0,T);U), \ S_T x_0 + \mathcal{B}^T u = 0 \right\}
\end{equation*}
the \emph{control cost in time $T$ via $L_r ((0,T);U)$} of the system \eqref{eq:System}. If $X$ and $U$ are Hilbert spaces and $r = 2$, it is well known that the control cost $C_T$ equals the smallest constant $C_{\mathrm{obs}}$ such that the system~\eqref{eq:AdjointSystem} satisfies a final state observability estimate. This fact is a direct consequence of Lemma~\ref{lemma:DR}.
\par
If $U$ is not a Hilbert space, or $r \not = 2$, the construction above does not apply directly, since it is not clear how to extend the operator $H$ to the whole space by keeping its relevant properties.  It is an open question if control costs can be estimated by the observability constant in the general setting. Under some additional assumption we can formulate the following lemma.
\begin{lemma} \label{lemma:CT<COBS}
 Let $X$ and $U$ be Banach spaces, $(S_t)_{t \geq 0}$ a $C_0$-semigroup on $X$, $-A$ the corresponding infinitesimal generator on $X$, $B \in \cL (U,X)$, $T > 0$, $r,r' \in (1,\infty)$ with $1/r' + 1/r = 1$, and $C_\mathrm{obs}>0$. Assume that the system~\eqref{eq:AdjointSystem} satisfies the final state observability estimate 
 \begin{equation} \label{eq:Dual-ObsEstimate}
\forall x' \in X' \colon \quad \lVert S'_T x' \rVert_{X'} 
\leq 
C_\mathrm{obs} \lVert B' S_{(\cdot)}' x'\rVert_{L_{r'}((0,T);U')}.
\end{equation}
Then the system \eqref{eq:System} is null-controllable in time $T$ via $L_r ((0,T) ; U)$. 
Moreover, there exists 
\[
 H\colon \overline{\ran {\mathcal{B}^T}'} \to X' \quad\text{such that}\quad
 H{\mathcal{B}^T}' = S_T' \quad \text{and} \quad
 \lVert H u' \rVert_{X'} \le C_\mathrm{obs} \lVert u' \rVert_{L_{r'} ((0,T);U')}  
\]
for all $u' \in \overline{\ran {\mathcal{B}^T}'}$, where $\mathcal{B}^T \colon L_r ((0,T);U) \to X$ is as in \eqref{eq:controllability-map}.
Suppose further that there is an extension $\tilde{H}\colon L_{r'}((0,T);U')\to X'$ of $H$ with $\lVert \tilde{H} u' \rVert \le C_\mathrm{obs} \lVert u'\rVert_{L_{r'}((0,T);U')}$ for all $u' \in L_{r'}((0,T);U')$. Then the control cost in time $T$ via $L_{r} ((0,T);U)$ of the system \eqref{eq:System} satisfies $C_T \le C_\mathrm{obs}$.
\end{lemma}
\begin{proof}[Proof of Lemma~\ref{lemma:CT<COBS}]
Equation~\eqref{eq:Dual-ObsEstimate} is equivalent to statement (b) of Lemma~\ref{lemma:DR} with $V=Z=X$, $W = L_r ((0,T);U)$, $F = S_T$, and $G = \mathcal{B}^T \in \cL (W , X)$ with $\mathcal{B}^T$ as in \eqref{eq:controllability-map}. The implication (b) $\Rightarrow$ (a) of Lemma~\ref{lemma:DR} implies that $\ran(S_T)\subset \ran(\mathcal{B}^T)$, i.e., null-controllability of the system \eqref{eq:System}. The implication (b) $\Rightarrow$ (c) of Lemma~\ref{lemma:DR} ensures the existence of the operator $H$ with the desired properties, which proves the first assertion.
\par
The dual operator of $\mathcal{B}^T$ is given by $({\mathcal{B}^T}' x') (t) = B' S_t' x'$ for $x'\in X'$. For an arbitrary initial state $x_0 \in X$ we choose the control function $u \in L_r ((0,T);U)$, $u (t) = (-\tilde{H}' x_0)(T-t)$, with $\tilde H$ as in the hypothesis of the lemma. Since $\tilde{H}{\mathcal{B}^T}' = S_T'$ by assumption, we obtain for all $x' \in X'$
\begin{align*}
\langle S_T x_0, x' \rangle_{X,X'} &= \langle \tilde{H}' x_0, {\mathcal{B}^T}'x' \rangle_{L_{r}((0,T);U),L_{r'}((0,T);U')} \\ &= -\int_0^T \langle u(T-t), B'S_t'x' \rangle_{U,U'} \drm t \\
&= -\int_0^T \langle S_t B u(T-t),x' \rangle_{X,X'} ~ \drm t = -\langle \mathcal{B}^T u , x' \rangle_{X,X'} .
\end{align*}
Thus, the solution of \eqref{eq:System} satisfies $x (T) = S_T x_0 + \mathcal{B}^T u = 0$. For the norm of the control function we have by assumption on $\tilde H$
\[
 \lVert u \rVert_{L_r ((0,T);U)} \leq \lVert \tilde{H} \rVert \lVert x_0 \rVert_X \leq C_{\mathrm{obs}} \lVert x_0 \rVert_X .
\]
This shows that the control cost in time $T$ via $L_r ((0,T);U)$ of the system \eqref{eq:System} satisfies $C_T \leq C_{\mathrm{obs}}$.
\end{proof}

From Lemma~\ref{lemma:CT<COBS} and Theorem~\ref{Thm:Observability_Elliptic_Operator_Rd} we obtain the following corollary. It complements Theorem~\ref{theorem:null<->thick} and provides an explicit upper bound on the control cost for elliptic operators $A$ and interior control on thick sets.
\begin{corollary}
  \label{Thm:Cost_Elliptic_Operator_Rd}
  Let $p,p',r, r'\in (1,\infty)$ such that $1/p' + 1/p = 1$ and $1/r' + 1/r = 1$, $a\colon\R^d \to \C$ a homogeneous strongly elliptic polynomial in $\R^d$ of degree $m\geq 2$, $A_p$ the associated elliptic operator in $L_p(\R^d)$, $(S_t)_{t\geq 0}$ the bounded $C_0$-semigroup on $L_p (\R^d)$ generated by $-A_p$, $\thickset \subset \R^d$ a $(\rho,L)$-thick set, and $T > 0$. Then the system 
  \begin{align} \label{eq:system:nullcontrol:elliptic}
	\dot{x}(t) & = -A_p x(t) + \1_\thickset' u(t), \quad t \in (0,T], \quad x(0) = x_0 \in L_p(\R^d)
  \end{align}
  is null-controllable in time $T$ via $L_{r} ((0,T);L_p(\thickset) )$. Moreover, there exists 
  \[
   H\colon \overline{\ran {\mathcal{B}^T}'} \to L_{p'}(\R^d)
   \quad\text{such that}\quad
   H{\mathcal{B}^T}' = S_T' 
  \]
  and $\lVert Hu' \rVert_{L_{p'}(\R^d)} \leq C_\mathrm{obs} \lVert u' \rVert_{L_{r'}((0,T);L_{p'}(\thickset))}$ for all $u' \in \overline{\ran {\mathcal{B}^T}'}$, where as in \eqref{eq:controllability-map} we have $\mathcal{B}^T \colon L_r ((0,T);L_p(\thickset)) \to L_p (\R^d)$ with $B = \1_\thickset'$, 
	  \begin{align*}
C_{\mathrm{obs}}
& = \frac{D_1 M^{16}}{T^{1/r'}} \left( \frac{K^d}{\rho} \right)^{D_2} \exp \left(\frac{D_3 (\lvert L \rvert_1 \ln (K^d / \rho))^{m/(m-1)}}{(c T)^{\frac{1}{m-1}}} \right), 
\end{align*}
where $K\geq 1$ is a universal constant, $D_1,D_2 \geq 1$ depending on $d$, $D_3 \geq 1$ depending on $d$ and $p'$, $M = \sup_{t \geq 0} \lVert S'_t \rVert$, and where $c > 0$ is such that $\re a (\xi) \geq c \lvert \xi \rvert^m$ for all $\xi \in \R^d$.
 Suppose further that there is an extension $\tilde{H}\colon L_{r'}((0,T);L_{p'}(\thickset))\to L_{p'}(\R^d)$ of $H$ with $\lVert \tilde{H}u' \rVert_{L_{p'}(\R^d)} \le C_\mathrm{obs} \lVert u' \rVert_{L_{r'}((0,T);L_{p'}(\thickset))}$ for all $u' \in L_{r'}((0,T);L_{p'}(\thickset))$. Then the control cost in time $T$ via $L_r ((0,T);L_p (\thickset))$ of the system \eqref{eq:system:nullcontrol:elliptic} satisfies $C_T \le C_{\mathrm{obs}}$.
\end{corollary}
\begin{remark}
  In general, it may be difficult to show the existence of an extension $\tilde{H}$ of $H$ as in Lemma~\ref{lemma:CT<COBS} and Corollary~\ref{Thm:Cost_Elliptic_Operator_Rd}.
  However, if $U$ is a Hilbert space (or $p = 2$) and $r = 2$, the existence is trivial, since we can choose $\tilde H = HP$ with a suitable orthogonal projection $P$.
\end{remark}
\section*{Acknowledgment}
The authors thank Clemens Bombach for valuable comments which helped to significantly improve an earlier version of this manuscript.
%
%
%

\begin{thebibliography}{NTTV20b}

\bibitem[AEWZ14]{ApraizEWZ-14}
J.~Apraiz, L.~Escauriaza, G.~Wang, and C.~Zhang.
\newblock Observability inequalities and measurable sets.
\newblock {\em J. Eur. Math. Soc. (JEMS)}, 16(11):2433--2475, 2014.

\bibitem[Bar14]{Barbu-14}
V.~Barbu.
\newblock Exact null internal controllability for the heat equation on
  unbounded convex domains.
\newblock {\em ESAIM Control Optim. Calc. Var.}, 20(1):222--235, 2014.

\bibitem[BPS18]{BeauchardP-18}
K.~Beauchard and K.~Pravda-Starov.
\newblock Null-controllability of hypoelliptic qua\-dra\-tic differential
  equations.
\newblock {\em J. \'Ec. polytech. Math.}, 5:1--43, 2018.

\bibitem[DE19]{DardeE-19}
J.~Dard\'{e} and S.~Ervedoza.
\newblock On the cost of observability in small times for the one-dimensional
  heat equation.
\newblock {\em Anal. PDE}, 12(6):1455--1488, 2019.

\bibitem[Dou66]{Douglas-66}
{R. G.} Douglas.
\newblock On majorization, factorization, and range inclusion of operators on
  hilbert space.
\newblock {\em Proc. Amer. Math. Soc.}, 2(17):413--415, 1966.

\bibitem[DR77]{DoleckiR-77}
S.~Dolecki and {D.~L.} Russell.
\newblock A general theory of observation and control.
\newblock {\em SIAM J. Control Optim.}, 2(15):185--220, 1977.

\bibitem[Egi]{Egidi-18-arxiv}
M.~Egidi.
\newblock On null-controllability of the heat equation on infinite strips and
  control cost estimate.
\newblock Math. Nachr., to appear.

\bibitem[EMZ15]{EscauriazaMZ-15}
L.~Escauriaza, S.~Montaner, and C.~Zhang.
\newblock Observation from measurable sets for parabolic analytic evolutions
  and applications.
\newblock {\em J. Math. Pures Appl.}, 104(5):837--867, 2015.

\bibitem[EV18]{EgidiV-18}
M.~Egidi and I.~Veseli{\'c}.
\newblock Sharp geometric condition for null-controllability of the heat
  equation on $\mathbb{R}^d$ and consistent estimates on the control cost.
\newblock {\em Arch. Math. (Basel)}, 111(1):1--15, 2018.

\bibitem[EV20]{EgidiV-16-arxiv}
M.~Egidi and I.~Veseli\'c.
\newblock Scale-free unique continuation estimates and {L}ogvi\-nenko-{S}ereda
  {T}heorems on the torus.
\newblock {\em Ann. Henri Poincar\'e}, 21(12):3757--3790, 2020.

\bibitem[EZ11]{ErvedozaZ-11}
S.~Ervedoza and E.~Zuazua.
\newblock Sharp observability estimates for heat equations.
\newblock {\em Arch. Ration. Mech. Anal.}, 202(3):975--1017, 2011.

\bibitem[FI96]{FursikovI-96}
{A.~V.} Fursikov and {O. Y.} Imanuvilov.
\newblock {\em Controllability of Evolution Equations}, volume~34 of {\em Suhak
  kang\v{u}irok}.
\newblock Seoul National University, Seoul, Korea, 1996.

\bibitem[FZ00]{FernandezZ-00}
E.~{Fern{\'a}ndez-Cara} and E.~Zuazua.
\newblock The cost of approximate controllability for heat equations: The
  linear case.
\newblock {\em Adv. Differential Equations}, 5(4--6):465--514, 2000.

\bibitem[Gd07]{GonzalezT-07}
M.~{Gonz{\'a}lez-Burgos} and L.~{de Teresa}.
\newblock Some results on controllability for linear and nonlinear heat
  equations in unbounded domains.
\newblock {\em Adv. Differential Equations}, 12(11):1201--1240, 2007.

\bibitem[G{\"{u}}i85]{Guichal-85}
{E.~N.} G{\"{u}}ichal.
\newblock A lower bound of the norm of the control operator for the heat
  equation.
\newblock {\em J. Math. Anal. Appl.}, 110(2):519--527, 1985.

\bibitem[Haa06]{Haase-06}
M.~Haase.
\newblock {\em The Functional Calculus for Sectorial Operators}, volume 169 of
  {\em Oper. Theory Adv. Appl.}
\newblock Birkh\"auser, Basel, 2006.

\bibitem[JL99]{JerisonL-99}
D.~Jerison and G.~Lebeau.
\newblock Nodal sets of sums of eigenfunctions.
\newblock In M.~Christ, {C.~E.} Kenig, and C.~Sadosky, editors, {\em Harmonic
  Analysis and Partial Differential Equations}, Chicago Lectures in
  Mathematics, pages 223--239. University of Chicago Press, Chicago, IL, 1999.

\bibitem[Kov00]{Kovrijkine-00}
O.~Kovrijkine.
\newblock {\em Some Estimates of {F}ourier Transforms}.
\newblock PhD thesis, California Institute of Technology, 2000.

\bibitem[Kov01]{Kovrijkine-01}
O.~Kovrijkine.
\newblock Some results related to the {L}ogvinenko-{S}ereda {T}heorem.
\newblock {\em Proc. Amer. Math. Soc.}, 129(10):3037--3047, 2001.

\bibitem[Lis12]{Lissy-12}
P.~Lissy.
\newblock A link between the cost of fast controls for the 1-d heat equation
  and the uniform controllability of a 1-d transport-diffusion equation.
\newblock {\em C. R. Math. Acad. Sci. Paris}, 350(11):591--595, 2012.

\bibitem[Lis15]{Lissy-15}
P.~Lissy.
\newblock Explicit lower bounds for the cost of fast controls for some 1-{D}
  parabolic or dispersive equations, and a new lower bound concerning the
  uniform controllability of the 1-{D} transport-diffusion equation.
\newblock {\em J. Differential Equations}, 259(10):5331--5352, 2015.

\bibitem[LL]{LaurentL-18-arxiv}
C.~Laurent and M.~L\'eautaud.
\newblock Observability of the heat equation, geometric constants in control
  theory, and a conjecture of {L}uc {M}iller.
\newblock Anal. PDE, to appear.

\bibitem[LL12]{LeRousseauL-12}
J.~{Le Rousseau} and G.~Lebeau.
\newblock On {C}arleman estimates for elliptic and parabolic operators.
  {A}pplications to unique continuation and control of parabolic equations.
\newblock {\em ESAIM Control Optim. Calc. Var.}, 18(3):712--747, 2012.

\bibitem[LM]{LebeauM-19-arxiv}
G.~Lebeau and I.~Moyano.
\newblock Spectral inequalities for the {S}chr{\"o}dinger operator.
\newblock 2019, arXiv:1901.03513.

\bibitem[LM16]{LeRousseauM-16}
J.~{Le Rousseau} and I.~Moyano.
\newblock Null-controllability of the {K}olmogorov equation in the whole phase
  space.
\newblock {\em J. Differential Equations}, 260(4):3193--3233, 2016.

\bibitem[LR95]{LebeauR-95}
G.~Lebeau and L.~Robbiano.
\newblock Contr{\^o}le exact de l'{\'e}quation de la chaleur.
\newblock {\em Comm. Partial Differential Equations}, 20(1--2):335--356, 1995.

\bibitem[LS74]{LogvinenkoS-74}
{V.~N.} Logvinenko and {Ju.~F.} Sereda.
\newblock Equivalent norms in spaces of entire functions of exponential type.
\newblock {\em Teor. Funkts., Funkts. anal. Prilozh.}, 20:102--111, 1974.

\bibitem[LZ98]{LebeauZ-98}
G.~Lebeau and E.~Zuazua.
\newblock Null-controllability of a system of linear thermoelasticity.
\newblock {\em Arch. Ration. Mech. Anal.}, 141(4):297--329, 1998.

\bibitem[Mil04a]{Miller-04}
L.~Miller.
\newblock Geometric bounds on the growth rate of null-controllability cost for
  the heat equation in small time.
\newblock {\em J. Differential Equations}, 204(1):202--226, 2004.

\bibitem[Mil04b]{Miller-04b}
L.~Miller.
\newblock How violent are fast controls for {S}chr\"{o}dinger and plate
  vibrations?
\newblock {\em Arch. Ration. Mech. Anal.}, 172(3):429--456, 2004.

\bibitem[Mil05a]{Miller-05}
L.~Miller.
\newblock On the null-controllability of the heat equation in unbounded
  domains.
\newblock {\em Bull. Sci. Math.}, 129(2):175--185, 2005.

\bibitem[Mil05b]{Miller-05b}
L.~Miller.
\newblock Unique continuation estimates for the laplacian and the heat equation
  on non-compact manifolds.
\newblock {\em Math. Res. Lett.}, 12(1):37--47, 2005.

\bibitem[Mil06a]{Miller-06}
L.~Miller.
\newblock The control transmutation method and the cost of fast controls.
\newblock {\em SIAM J. Control Optim.}, 45(2):762--772, 2006.

\bibitem[Mil06b]{Miller-06b}
L.~Miller.
\newblock On exponential observability estimates for the heat semigroup with
  explicit rates.
\newblock {\em Atti Accad. Naz. Lincei Rend. Lincei Mat. Appl.},
  17(4):351--366, 2006.

\bibitem[Mil10]{Miller-10}
L.~Miller.
\newblock A direct {L}ebeau-{R}obbiano strategy for the observability of
  heat-like semigroups.
\newblock {\em Discrete Contin. Dyn. Syst. Ser. B}, 14(4):1465--1485, 2010.

\bibitem[NTTV18]{NakicTTV-18}
I.~Naki\'c, M.~T\"aufer, M.~Tautenhahn, and I.~Veseli\'c.
\newblock Scale-free unique continuation principle, eigenvalue lifting and
  {W}egner estimates for random {S}chr\"odinger operators.
\newblock {\em Anal. PDE}, 11(4):1049--1081, 2018.

\bibitem[NTTV20a]{NakicTTV-18b-arxiv}
I.~Naki\'c, M.~T\"aufer, M.~Tautenhahn, and I.~Veseli\'c.
\newblock Sharp estimates and homogenization of the control cost of the heat
  equation on large domains.
\newblock {\em ESAIM Control Optim. Calc. Var.}, 26(54):26 pages, 2020.

\bibitem[NTTV20b]{NakicTTV-18-arxiv}
I.~Naki\'c, M.~T\"aufer, M.~Tautenhahn, and I.~Veseli\'c.
\newblock Unique continuation and lifting of spectral band edges of
  {S}chr\"odinger operators on unbounded domains.
\newblock {\em J. Spectr. Theory}, 10(3):843--885, 2020.
\newblock With an appendix by Albrecht Seelmann.

\bibitem[Phu04]{Phung-04}
{K.~D.} Phung.
\newblock Note on the cost of the approximate controllability for the heat
  equation with potential.
\newblock {\em J. Math. Anal. Appl.}, 295(2):527--538, 2004.

\bibitem[Phu18]{Phung-18}
{K.~D.} Phung.
\newblock Carleman commutator approach in logarithmic convexity for parabolic
  equations.
\newblock {\em Math. Control Relat. Fields}, 8(3{\&}4):899--933, 2018.

\bibitem[PW13]{PhungW-13}
Kim~Dang Phung and Gengsheng Wang.
\newblock An observability estimate for parabolic equations from a measurable
  set in time and its applications.
\newblock {\em J. Eur. Math. Soc. (JEMS)}, 15(2):681--703, 2013.

\bibitem[Sei84]{Seidman-84}
{T.~I.} Seidman.
\newblock Two results on exact boundary control of parabolic equations.
\newblock {\em Appl. Math. Optim.}, 11(2):145--152, 1984.

\bibitem[TT07]{TenenbaumT-07}
G.~Tenenbaum and M.~Tucsnak.
\newblock New blow-up rates for fast controls of {S}chr{\"o}dinger and heat
  equations.
\newblock {\em J. Differential Equations}, 243(1):70--100, 2007.

\bibitem[TT11]{TenenbaumT-11}
G.~Tenenbaum and M.~Tucsnak.
\newblock On the null-controllability of diffusion equations.
\newblock {\em ESAIM Control Optim. Calc. Var.}, 17(4):1088--1100, 2011.

\bibitem[Vie05]{Vieru-05}
A.~Vieru.
\newblock On null controllability of linear systems in {B}anach spaces.
\newblock {\em Systems Control Lett.}, 54(4):331--337, 2005.

\bibitem[WWZZ19]{WangWZZ-19}
G.~Wang, M.~Wang, C.~Zhang, and Y.~Zhang.
\newblock Observable set, observability, interpolation inequality and spectral
  inequality for the heat equation in $\mathbb{R}^n$.
\newblock {\em J. Math. Pures Appl.}, 126:144--194, 2019.

\bibitem[WZ17]{WangZ-17}
G.~Wang and C.~Zhang.
\newblock Observability inequalities from measurable sets for some abstract
  evolution equations.
\newblock {\em SIAM J. Control Optim.}, 55(3):1862--1886, 2017.

\bibitem[YLC06]{YuLC-06}
X.~Yu, K.~Liu, and P.~Chen.
\newblock On null controllability of linear systems via bounded control
  functions.
\newblock In {\em in Proceedings of the 2006 American Control Conference},
  pages 1458--1461. IEEE, Piscataway, 2006.

\end{thebibliography}
%
%

%
\end{document}